\documentclass[12pt, reqno]{amsart}

\usepackage{
amsmath,
amssymb,
amsthm,
mathrsfs,
amsfonts,
ytableau,
enumitem,
comment,
upgreek,
soul,
yhmath,
mathtools,
shuffle,
kotex,
array,
caption,
hhline 
}

\usepackage[linktocpage]{hyperref}
\hypersetup{
    colorlinks=true,
    linkcolor = blue,
    citecolor=magenta,
    urlcolor=cyan,
}

\usepackage{float}
\usepackage[capitalise, nameinlink]{cleveref}
\crefname{equation}{}{}
\crefname{subsection}{Subsection}{Subsections}
\crefname{example}{Example}{Examples}
\crefname{theorem}{Theorem}{Theorems}
\crefname{lemma}{Lemma}{Lemmas}
\crefname{proposition}{Proposition}{Propositions}
\crefname{thm}{Theorem}{Theorems}
\crefname{lem}{Lemma}{Lemmas}
\crefname{prop}{Proposition}{Propositions}
\crefname{figure}{Figure}{Figures}
\crefname{fig}{Figure}{Figures}
\crefname{remark}{Remark}{Remarks}
\crefname{rem}{Remark}{Remarks}
\crefname{cor}{Corollary}{Corollaries}
\crefname{corollary}{Corollary}{Corollaries}
\crefname{conjecture}{Conjecture}{Conjectures}
\crefname{conj}{Conjecture}{Conjectures}
\crefname{ex}{Example}{Examples}

\usepackage{tikz,tikz-cd}
\usetikzlibrary{arrows.meta}
\tikzset{
  long mapsto/.style={
    -{Straight Barb[angle'=60, scale=1.5]},
    shorten <=2pt,
    shorten >=2pt,
  }
}
\usepackage[colorinlistoftodos, textwidth = 2.3cm]{todonotes}

\usepackage{bbm}
\usepackage{dsfont}

\ytableausetup{mathmode, boxsize=1.2em}

\title
[Representation Theory of 
$0$-Schur Algebras and Related Categories]
{Representation Theory of 
$0$-Schur Algebras and Related Categories}

\author[W.-S. Jung]{Woo-Seok Jung}
\address{Department of Mathematics, University of Seoul, Seoul 02504, Republic of Korea}
\email{jungws@uos.ac.kr}

\author[Y.-T. Oh]{Young-Tak Oh}
\address{Department of Mathematics/Institute for Mathematical and Data Sciences, Sogang University, Seoul 04107, Republic of Korea}
\email{ytoh@sogang.ac.kr}

\thanks{
}

\date{\today}

\newtheorem{theorem}{Theorem}[section]
\newtheorem{proposition}[theorem]{Proposition}
\newtheorem{lemma}[theorem]{Lemma}
\newtheorem{corollary}[theorem]{Corollary}

\newtheorem*{claim*}{Claim}

\theoremstyle{definition}

\newtheorem{example}[theorem]{Example}
\newtheorem{definition}[theorem]{Definition}
\newtheorem{remark}[theorem]{Remark}

\numberwithin{equation}{section} \numberwithin{figure}{section}
\numberwithin{table}{section}

\def\C{\mathbb C}

\newcommand{\nc}{\newcommand}

\nc{\wt}{\mathsf{wt}}
\nc{\sw}{\mathsf{sw}}
\nc{\sDes}{\mathrm{sDes}}
\nc{\Bre}{\mathrm{Bre}}
\nc{\SG}{\mathfrak{S}}
\nc{\frakR}{\mathfrak{R}}
\nc{\frakL}{\mathfrak{L}}
\nc{\PCT}{\mathrm{PCT}}
\nc{\SPCT}{\mathrm{SPCT}}
\nc{\RT}{\mathrm{RT}}
\nc{\SRT}{\mathrm{SRT}}
\nc{\RCT}{\mathrm{RCT}}
\nc{\SRCT}{\mathrm{SRCT}}
\nc{\SYRT}{\mathrm{SYRT}}
\nc{\SYCT}{\mathrm{SYCT}}
\nc{\SPYCT}{\mathrm{SPYCT}}
\nc{\tst}{\mathtt{st}}
\nc{\Span}{\mathrm{span}}
\nc{\comp}{\mathrm{comp}}
\nc{\rmst}{\mathrm{st}}
\nc{\std}{\mathsf{std}}
\nc{\Des}{\mathrm{Des}}
\nc{\set}{\mathrm{set}}
\nc{\cf}{\textsf{cf}}
\nc{\scf}{\textsf{scf}}
\nc{\ch}{\mathrm{ch}}
\nc{\cl}{\mathrm{cl}}
\nc{\id}{\mathrm{id}}
\nc{\sh}{\mathrm{sh}}
\nc{\Cop}{\mathrm{Cop}}
\nc{\bfS}{\mathbf{S}}
\nc{\bfm}{\mathbf{m}}
\nc{\hbfS}{\widehat{\mathbf{S}}}
\nc{\bfF}{\mathbf{F}}
\nc{\calB}{\mathcal{B}}
\nc{\calS}{\mathcal{S}}
\nc{\hcalS}{\widehat{\mathcal{S}}}
\nc{\alphamax}{\alpha_{\rm max}}
\nc{\brho}{\overline{\rho}}
\nc{\bphi}{\overline{\phi}}
\nc{\calV}{\mathcal{V}}
\nc{\calR}{\mathcal{R}}
\nc{\sfR}{\mathsf{R}}
\nc{\calG}{\mathcal{G}}
\nc{\tal}{\lambda(\alpha)}
\nc{\tbe}{\widetilde{\beta}}
\nc{\opi}{\overline{\pi}}
\nc{\calP}{\mathcal{P}}
\nc{\rmtop}{\mathrm{top}}
\nc{\rad}{\mathrm{rad}}
\nc{\bfP}{\mathbf{P}}
\nc{\SET}{\mathrm{SET}}
\nc{\SIT}{\mathrm{SIT}}
\nc{\rev}{\mathrm{r}}
\nc{\Th}{\theta}
\nc{\mPhi}{\Phi}
\nc{\mphi}{\phi}
\nc{\mPsi}{\Psi}
\nc{\hmPsi}{\widehat{\Psi}}
\nc{\mpsi}{\psi}
\nc{\mGam}{\Gamma}
\nc{\tcd}{\mathtt{cd}}
\nc{\trd}{\mathtt{rd}}
\nc{\trcd}{\mathtt{rcd}}
\nc{\rmr}{\mathrm{r}}
\nc{\rmc}{\mathrm{c}}
\nc{\rmt}{\mathrm{t}}
\nc{\bubact}{\,\scalebox{0.6}{$\bullet$}\,}
\nc{\hbubact}{\,\scalebox{0.6}{$\widehat{\bullet}$}\,}
\nc{\col}{\mathsf{col}}
\nc{\row}{\mathsf{row}}
\nc{\calE}{\mathcal{E}}
\nc{\calT}{\mathscr{T}}
\nc{\sfT}{\mathsf{T}}
\nc{\calEsa}{\mathcal{E}^\sigma(\alpha)}
\nc{\tauC}{\tau_{\scalebox{0.5}{$C$}}}
\nc{\sytabC}{\sytab_{\scalebox{0.5}{$C$}}}
\nc{\bbfP}{\overline{\bfP}}
\nc{\pr}{\mathbf{pr}}
\nc{\Ups}{\Upsilon}
\nc{\pact}{\diamond}
\nc{\tauE}{\tau_{\scalebox{0.5}{$E$}}}
\nc{\tauF}{\tau_{\scalebox{0.5}{$F$}}}
\nc{\tauG}{\tau_{\scalebox{0.5}{$G$}}}
\nc{\rtE}{T_{\scalebox{0.5}{$E$}}}
\nc{\rtF}{T_{\scalebox{0.5}{$F$}}}
\nc{\rtG}{T_{\scalebox{0.5}{$G$}}}
\nc{\oPaE}{\overline{\Phi}_{\alpha_E}}
\nc{\oPaF}{\overline{\Phi}_{\alpha_F}}
\nc{\oPaG}{\overline{\Phi}_{\alpha_G}}
\nc{\tab}{\tau}
\nc{\sytab}{\widehat{\tau}}
\nc{\hatE}{\widehat{E}}
\nc{\hati}{\hat{i}}
\nc{\hcalE}{\widehat{\calE}}
\nc{\hatC}{\widehat{C}}
\nc{\bal}{{\boldsymbol{\upalpha}}}
\nc{\bbe}{{\boldsymbol{\upbeta}}}

\nc{\bgam}{{\boldsymbol{\upgamma}}}
\nc{\bdel}{{\boldsymbol{\updelta}}}
\nc{\weakcon}{\odot}
\nc{\basisI}{I}

\nc{\ldalpha}{\lambda(\alpha)}
\nc{\SRIT}{\mathrm{SRIT}}
\nc{\re}{\mathrm{rev}}
\nc{\otau}{\overline{\tau}}
\nc{\rtop}{{\rm top}}
\nc{\sfc}{\mathsf{c}}
\nc{\sfr}{\mathsf{r}}
\nc{\tH}{\mathtt{H}}
\nc{\tV}{\mathtt{V}}
\nc{\rpi}{\mathring{\pi}}
\nc{\cpi}{\check{\pi}}
\nc{\frakm}{\mathfrak{m}}
\nc{\fke}{\mathfrak{e}}
\nc{\Hom}{\mathrm{Hom}}
\nc{\module}{\mathrm{mod} \, }

\nc{\SPCTsa}{\SPCT^\sigma(\alpha)}
\nc{\bfSsa}{\bfS_\alpha^\sigma}
\nc{\bfSsaC}{{\bfS}^\sigma_{\alpha,C}}
\nc{\hbfSsa}{\widehat{\bfS}_\alpha^\sigma}
\nc{\upineq}{\rotatebox{90}{$<$}}
\nc{\downineq}{\rotatebox{270}{$<$}}
\nc{\diagineq}{\rotatebox{135}{$<$}}
\nc{\frakB}{\mathfrak{B}}
\nc{\hxi}{\widehat{\xi}}
\nc{\hxidwJ}{\hxi_{\scalebox{0.55}{$J$}}}
\nc{\hxiupJ}{\hxi^{\scalebox{0.55}{$J$}}}
\nc{\scrS}{\mathscr{S}}
\nc{\bfT}{\mathbf{T}}

\nc{\ra}{\rightarrow}

\nc{\matr}[2]{\left( \hspace{-1ex} \begin{array}{c} #1 \\ #2 \end{array} \hspace{-1ex} \right)}

\definecolor{wsgreen}{rgb}{0,0.5,0}

\nc{\DIRT}{\mathrm{DIRT}}
\nc{\hpi}{\pi}
\nc{\frakI}{\mathfrak{I}}
\nc{\hfrakI}{\widehat{\mathfrak{I}}}
\nc{\orho}{\overline{\rho}}
\nc{\autotheta}{\uptheta}
\nc{\autophi}{\upphi}
\nc{\autochi}{\upchi}
\nc{\autoomega}{\upomega}
\nc{\hIM}{\widehat{\frakB}}
\nc{\bfpi}{\boldsymbol{\uppi}}
\nc{\bfopi}{\overline{\boldsymbol{\uppi}}}
\nc{\ofrakB}{\overline{\frakB}}
\nc{\rmw}{\mathrm{w}}
\nc{\ostar}{\;\overline{*}\;}
\nc{\rank}{\mathrm{rank}}
\nc{\fkp}{\mathfrak{p}}
\nc{\bfR}{\mathbf{R}}
\nc{\ro}{\mathsf{row}}
\nc{\co}{\mathsf{col}}
\nc{\cb}{\mathsf{CB}}
\nc{\End}{\text{End}}

\nc{\upsig}{{\boldsymbol{\upsigma}}}
\nc{\bfSsaE}{{\bfS}^\upsig_{\alpha,E}}

\nc{\hfkp}{\widehat{\mathfrak{p}}}

\nc{\hautophi}{{\widehat{\autophi}}}
\nc{\hautotheta}{{\widehat{\autotheta}}}
\nc{\hautoomega}{{\widehat{\autoomega}}}
\nc{\rmperm}{\mathrm{perm}}

\nc{\bfsigJ}{\boldsymbol{\sigma}_{\scalebox{0.55}{$J$}}}
\nc{\bfrhoJ}{\boldsymbol{\rho}^{\scalebox{0.55}{$J$}}}

\nc{\pistar}[1]{\pi_{#1}^*}
\nc{\wfkp}{\widetilde{\mathfrak{p}}}
\nc{\bfpsi}{\boldsymbol{\uppsi}}

\nc{\yt}[1]{\todo[size=\tiny,color=blue!10]{#1 \\ \hfill --- Young-Tak}}
\nc{\YT}[1]{\todo[size=\tiny,inline,color=blue!10]{#1
		\\ \hfill --- Young-Tak}}

\nc{\ws}[1]{\todo[size=\tiny,color=green!10]{#1 \\ \hfill ---  Woo-Seok}}
\nc{\WS}[1]{\todo[size=\tiny,inline,color=green!10]{#1
		\\ \hfill --- Woo-Seok}}

\definecolor{purple}{rgb}{0.44, 0.0, 1.0}

\newenvironment{red}{\relax\color{red}}{\hspace*{.5ex}\relax}
\newenvironment{blue}{\relax\color{blue}}{\hspace*{.5ex}\relax}
\newenvironment{green}{\relax\color{wsgreen}}{\hspace*{.5ex}\relax}
\newenvironment{magenta}{\relax\color{magenta}}{\hspace*{.5ex}\relax}
\newenvironment{purple}{\relax\color{purple}}{\hspace*{.5ex}\relax}

\nc{\ber}{\begin{red}}
\nc{\er}{\end{red}}
\nc{\beb}{\begin{blue}}
\nc{\eb}{\end{blue}}
\nc{\bema}{\begin{magenta}}
\nc{\ema}{\end{magenta}}
\nc{\begr}{\begin{green}}
\nc{\egr}{\end{green}}

\nc{\bepu}{\begin{purple}}
\nc{\epu}{\end{purple}}
\nc{\lb}{\pmb{\left[\vphantom{\frac{1}{2}}\right.}}
\nc{\rb}{\pmb{\left.\vphantom{\frac{1}{2}}\right]}}
\nc{\slb}{\pmb{[\vphantom{\frac{1}{2}}}}
\nc{\srb}{\pmb{\vphantom{\frac{1}{2}}]}}

\nc{\tran}{\mathsf{tran}}
\nc{\e}{\mathsf{e}}
\nc{\Int}{\mathsf{Int}}

\begin{document}

\maketitle

\begin{abstract}
Jensen, Su, and Yang described the projective indecomposable modules of the $0$-Schur algebra $\mathbf{S}_0(n,r)$ using its geometric realization. 
In this paper, the simple modules of $\mathbf{S}_0(n,r)$ are identified by computing the tops of the projective indecomposable modules.
Furthermore, functorial relations among the module categories
$\mathbf{H}_r(0)$\textsf{-mod},
$\mathbf{S}_0(n,r)$\textsf{-mod}, and
$U_0(\mathfrak{gl}_n)$\textsf{-mod}
are examined, where $\mathbf{H}_r(0)$ denotes the $0$-Hecke algebra and $U_0(\mathfrak{gl}_n)$ denotes the degenerate quantum group.
\end{abstract}

\tableofcontents

\section{Introduction}
Let $n$ and $r$ be nonnegative integers, and let $q$ be an indeterminate.
In 1989, Dipper and James~\cite{DJ89} introduced the quantum Schur algebra $S_q(n,r)$ as the endomorphism algebra 
$$
S_q(n,r)=\operatorname{End}_{H_r(q)} \!\left( \bigoplus_{\lambda \in \Lambda(n,r)} x_{\lambda} H_r(q) \right),
$$
where $H_r(q)$ is the Hecke algebra of type $A$ over $\mathbb{Z}[q]$ (see \cref{0-Schur algebra definition}). 
It plays a central role in the representation theory of Hecke algebras and quantum groups, being closely connected with Schur–Weyl duality.
The specialization of $S_q(n,r)$ at $q=0$ gives rise to the $0$-Schur algebra $S_0(n,r)$.

Beilinson, Lusztig, and MacPherson~\cite{BLM90} constructed certain finite-dimensional quotients of the quantized enveloping algebra $U_q(\mathfrak{gl}_n)$ using double flag varieties.
Du~\cite{Du95} subsequently proved that these quotients are isomorphic to the $q$-Schur algebras $S_q(n,\cdot)$, thereby providing a geometric realization of $S_q(n,r)$.
In parallel, Jensen and Su~\cite{JS15} established a geometric realization of the $0$-Schur algebra $S_0(n,r)$.

Let $\mathbf{S}_0(n, r) := S_0(n, r) \otimes_{\mathbb{Z}[q]} \mathbb{C}$ and  $\mathbf{H}_0(r) := H_0(r) \otimes_{\mathbb{Z}[q]} \mathbb{C}$.
Jensen, Su, and Yang~\cite{JSY16} employed the realization in~\cite{JS15} to construct its projective indecomposable $\mathbf{S}_0(n, r)$- modules.
More precisely, they provided a complete set $\{ P_{\lambda} \mid \lambda \in \Lambda^{\bullet}(n,r) \}$ of representatives of the isomorphism classes of projective indecomposable  $\mathbf{S}_0(n, r)$-modules, where $\Lambda^{\bullet}(n,r)$ is the set of weak compositions \( (\lambda_1, \lambda_2, \dots, \lambda_n)\) such that if \( \lambda_i = 0 \) for some \( i \), then \( \lambda_j = 0 \) for all \( j > i \).

In \cref{main result 1}, we construct simple $\mathbf{S}_0(n,r)$-modules $\{S_{\lambda} \mid \lambda \in \Lambda^{\bullet}(n,r)\}$ by determining the tops of the projective indecomposable modules (\cref{thm: Slambda is simple}).
We present a basis of $S_{\lambda}$ parametrized by {\em column block diagonal matrices}, and we show that this basis has several notable properties (\cref{lem: for main thm}).
As an immediate consequence of \cref{thm: Slambda is simple}, we obtain an explicit description of the radicals of projective indecomposable $\mathbf{S}_0(n,r)$-modules (\cref{Describing radical }).

In \cref{main result: 2}, we investigate functorial relations among the module categories $U_0(\mathfrak{gl}_n)\textsf{-mod}$, $\mathbf{S}_0(n,r)\textsf{-mod}$, and $\mathbf{H}_r(0)\textsf{-mod}$. Here, $U_0(\mathfrak{gl}_n)$ denotes the degenerate analogue of the quantum group $U_q(\mathfrak{gl}_n)$ introduced by Krob and Thibon~\cite{KT5}.
Letting $V_0 := \mathbb{C}^n$, the tensor space ${V_0}^{\otimes r}$ carries the structure of a $(U_0(\mathfrak{gl}_n), \mathbf{H}_r(0))$-bimodule, which defines the functor
\begin{equation*}
F_{n,r} : \mathbf{H}_r(0)\textsf{-mod} \longrightarrow U_0(\mathfrak{gl}_n)\textsf{-mod},
\quad M \mapsto {V_0}^{\otimes r} \otimes_{\mathbf{H}_r(0)} M.
\end{equation*}
By virtue of the isomorphism
$\mathbf{S}_0(n,r)\cong \operatorname{End}_{\mathbf{H}_r(0)}\!\left({V_0}^{\otimes r}\right)$,
we also define the functor
$$
G_{n,r} : \mathbf{H}_r(0)\textsf{-mod} \longrightarrow \mathbf{S}_0(n,r)\textsf{-mod}, 
\quad M \mapsto {V_0}^{\otimes r} \otimes_{\mathbf{H}_r(0)} M.
$$
Moreover, because the image of the representation of $U_0(\mathfrak{gl}_n)$ in $\operatorname{End}_{\mathbb{C}}({V_0}^{\otimes r})$ is contained in the subalgebra $\operatorname{End}_{\mathbf{H}_r(0)}({V_0}^{\otimes r})$, we obtain the functor
$$
\Psi_{n,r}: \mathbf{S}_0(n,r)\textsf{-mod} \longrightarrow U_0(\mathfrak{gl}_n)\textsf{-mod}.$$
These functors satisfy the relation $\Psi_{n,r} \circ G_{n,r} \cong  F_{n,r}$.

In~\cite{KT5}, Krob and Thibon studied the images of simple and projective $\mathbf{H}_r(0)$-modules under the functor $F_{n,r}$.
Here, we investigate the images of simple and projective modules under the functors $G_{n,r}$ and $\Psi_{n,r}$.
To begin, we recall the simple and projective indecomposable $\mathbf{H}_r(0)$-modules, which are naturally indexed by the strong compositions of $r$.
For each strong composition $\alpha$ of $r$, we denote by $F_\alpha$ and $R_\alpha$ the simple and projective indecomposable $\mathbf{H}_r(0)$-modules corresponding to $\alpha$
(for details, see \cref{subsec: 0-Hecke alg and QSym}).
As a first step, for each $\lambda \in \Lambda^{\bullet}(n,r)$, 
we establish a $U_0(\mathfrak{gl}_n)$-module isomorphism 
\[
\Psi_{n,r}(S_{\lambda}) \cong {V_0}^{\otimes r} \otimes_{\mathbf{H}_r(0)} F_{\lambda^{+}}
\]
(\cref{thm: 0-Schur functor degenerate quantum group}).
Building on this result, for each strong composition of $r$ with $\ell(\alpha)\le n$, we obtain $\mathbf{S}_0(n,r)$-module isomorphisms 
\[
G_{n,r}(F_{\alpha}) \cong S_{\alpha^{\bullet}}
\quad \text{ and }\quad 
G_{n,r}(R_{\alpha}) \cong P_{\alpha^{\bullet}},
\]
where $\alpha^{\bullet} \in \Lambda^{\bullet}(n, r)$ is the composition obtained by appending zeros to $\alpha$ so that its length becomes $n$ 
(\cref{cor: 0-Schur functor 0-Hecke}).
By combining the second isomorphism with the relation  $\Psi_{n,r} \circ G_{n,r} \cong  F_{n,r}$, we deduce that, for each $\lambda \in \Lambda^{\bullet}(n,r)$, 
there exists a $U_0(\mathfrak{gl}_n)$-module isomorphism 
\[ \Psi_{n,r}(P_{\lambda}) \cong {V_0}^{\otimes r} \otimes_{\mathbf{H}_r(0)} R_{\lambda^{+}},
\]
where $\lambda^+$ is the strong composition obtained by removing all zero entries from $\lambda$
(\cref{thm: 0-Schur functor degenerate quantum group PIM}).
We close this section by establishing the existence of a functor
\[
\xi_{n,r}: \mathcal C \to \mathbf{S}_0(n,r)\text{-\textsf{mod}}
\]
satisfying \(\Psi_{n,r}\circ \xi_{n,r}= \mathrm{id}\),
where \(\mathcal C\) denotes the full subcategory of \(U_0(\mathfrak{gl}_n)\text{-\textsf{mod}}\) consisting of polynomial \(U_0(\mathfrak{gl}_n)\)-modules of degree \(r\) (\cref{polynomial modules are Schur modules}).

In \cref{main result: 3}, we discuss the symmetries of the Cartan matrix of ${\mathbf H}_0(n,r)$ and its relation to the Cartan matrix of ${\mathbf S}_0(n,r)$.
By applying (anti)involution twists, we first show that the Cartan matrix of ${\mathbf H}_0(n,r)$ is constant on each $D_4 \times \mathbb{Z}_2$-orbit
(\cref{symmetry of Catran matrix of 0 Hecke algebra}).
Next, we prove that for $0 \le n \le r$, the Cartan matrix of ${\mathbf S}_0(n,r)$
can be obtained from that of ${\mathbf H}_r(0)$ by restricting the index set from all strong compositions of $r$ to $\Lambda^{\bullet}(n,r)$
(\cref{symmetries of 0 Schur cartan}).

\section{Preliminaries}\label{Section: Preliminaries}
Throughout this section, let \( r \) be a nonnegative integer.

\subsection{Weak and strong compositions}
A \emph{weak composition} of $r$ is a finite sequence of nonnegative integers $\lambda = (\lambda_1, \lambda_2, \ldots, \lambda_n)$ such that
$\sum_{i=1}^{n} \lambda_i = r$.
Each $\lambda_i$ is called a \emph{part} of $\lambda$. The integer $\ell(\lambda) := n$ is referred to as the \emph{length} of $\lambda$, and $|\lambda| := r$ as the \emph{size} of $\lambda$. The \emph{empty composition} $\varnothing$ is defined to be the unique weak composition of size and length zero. Let $\Lambda(n, r)$ denote the set of weak compositions of $r$ of length $n$.

A weak composition $\lambda = (\lambda_1, \lambda_2, \ldots, \lambda_n)$ of $r$ is called a \emph{strong composition} of $r$ if each part $\lambda_i$ is a positive integer. 
Let $\Lambda^+(n,r)$ denote the set of strong compositions of $r$ of length $n$ and set 
$$\Lambda^+(r):=\bigsqcup_{0\le n \le r} \Lambda^+(n,r).$$

Given $\lambda = (\lambda_1, \lambda_2, \ldots, \lambda_n) \in \Lambda^+(r)$ and $I = \{i_1 < i_2 < \cdots < i_k\} \subset [r-1]:= \{1, 2, \ldots, r-1\}$, define
$$
\set(\lambda) := \left\{\lambda_1,\ \lambda_1 + \lambda_2,\ \ldots,\ \lambda_1 + \lambda_2 + \cdots + \lambda_{n-1} \right\} \subset [r-1],
$$
and
$$
\comp(I) := (i_1,\ i_2 - i_1,\ i_3 - i_2,\ \ldots,\ r - i_k).
$$
Then the set of strong compositions of $r$ is in bijection with the power set of $[r-1] $ via the correspondence $\lambda \mapsto \set(\lambda)$ (or $I \mapsto \comp(I)$).

Let $\Lambda^{\bullet}(n, r) \subset \Lambda(n, r)$ denote the set of compositions $(\lambda_1, \lambda_2, \dots, \lambda_n)$ such that $\lambda_i = 0$ implies $\lambda_j = 0$ for all $j > i$.
For $\lambda \in \Lambda^{\bullet}(n, r)$, let $\lambda^+$ denote the strong composition obtained by removing all zero entries from $\lambda$.
For example, $(3,1,2,0,0,0)^+ = (3,1,2)$.
Conversely, for $\alpha \in \Lambda^+(r)$ with $\ell(\alpha) \le n$, let $\alpha^{\bullet} \in \Lambda^{\bullet}(n, r)$ be the composition obtained by appending zeros to $\alpha$ so that its length becomes $n$.
If $\ell(\alpha) > n$, we set $\alpha^{\bullet} = \varnothing$.
Then there are natural bijections
\begin{equation}\label{relation between weak and strong compositions}
\begin{aligned} 
& \Lambda^{\bullet}(n,r) \xrightarrow{\ \sim\ } \bigsqcup_{0 \le \ell \le n} \Lambda^+(\ell, r), \quad \lambda \mapsto \lambda^+,\\
& \bigsqcup_{0 \le \ell \le n} \Lambda^+(\ell, r) \xrightarrow{\ \sim\ } \Lambda^{\bullet}(n,r), \quad \alpha \mapsto \alpha^{\bullet}.
\end{aligned}
\end{equation}

Let $\alpha = (\alpha_1, \alpha_2, \ldots, \alpha_n) \in \Lambda^+(r)$. Define the \emph{reverse composition} $\alpha^{\mathrm{r}}$ of $\alpha$ by
$$
\alpha^{\mathrm{r}} := (\alpha_n, \alpha_{n-1}, \ldots, \alpha_1).
$$
The \emph{complement} $\alpha^{\mathrm{c}}$ of $\alpha$ is the unique composition of $r$ satisfying
$$
\set(\alpha^{\mathrm{c}}) = [r-1] \setminus \set(\alpha).
$$
The \emph{conjugate composition} $\alpha^{\mathrm{t}}$ of $\alpha$ is defined by
$$
\alpha^{\mathrm{t}} := (\alpha^{\mathrm{r}})^{\mathrm{c}} = (\alpha^{\mathrm{c}})^{\mathrm{r}}.
$$

\subsection{The $0$-Schur algebra $S_0(n,r)$}\label{0-Schur algebra definition}
Let $\SG_r$ be the symmetric group on $\{1,2, \ldots, r\}$
and $q$ be an indeterminate.
The Hecke algebra $H_r(q)$ associated with $\SG_r$ is the $\mathbb{Z}[q]$-algebra with unity generated by $T_i$ ($i=1,2, \ldots, r-1$) subject to the following defining relations:
\begin{align*}
    \begin{cases}
        T^2_i = (q-1) T_i + q, & \text{for } 1 \le i \le r-1, \\
        T_i T_j = T_j T_i, & \text{for } |i-j| > 1, \\
        T_i T_{i+1} T_i = T_{i+1} T_i T_{i+1}, & \text{for } 1 \le i < r-1.
    \end{cases}
\end{align*}
Let \( s_i \) denote the simple transposition \( (i \,\, i+1) \) for \( 1 \leq i \leq r-1 \). 
For $w \in \SG_r$, choose a reduced expression $w = s_{i_1} \cdots s_{i_t}$ and define $T_w := T_{i_1} \cdots T_{i_t}$.
It is well known that $T_w$ is independent of the choice of reduced expressions. Moreover, \( H_r(q) \) is a free \( \mathbb{Z}[q] \)-algebra with basis \( \{ T_w \mid w \in \SG_r \} \).

For a weak composition \( \lambda = (\lambda_1, \lambda_2, \ldots, \lambda_n) \in \Lambda(n,r) \), the corresponding \emph{Young subgroup} \( \SG_{\lambda} \) of \( \SG_r \) is defined by     
\[
\SG_{\lambda_1} \times \SG_{\lambda_2} \times \cdots \times \SG_{\lambda_n}.
\]
Define   
\[
x_{\lambda} := \sum_{w \in \SG_{\lambda}} T_w \in H_r(q).
\]
Following Dipper and James \cite{DJ89}, the \( q \)-Schur algebra \( S_q(n,r) \) is defined as the endomorphism algebra over $\mathbb{Z}[q]$
\[
S_q(n,r) := \operatorname{End}_{H_r(q)} \left( \bigoplus_{\lambda \in \Lambda(n,r)} x_{\lambda} H_r(q) \right).
\]
The \( 0 \)-Schur algebra \( S_0(n,r)\) arises as a degeneration of \( S_q(n,r) \) at \( q=0 \), given by  
\[
S_0(n,r) = S_q(n,r) \otimes_{\mathbb{Z}[q]} \mathbb{Z}
\cong S_q(n,r) \otimes_{\mathbb{Z}[q]} \mathbb{Z}[q]/(q).
\]  
This algebra has been studied in various works, including 
\cite{D98q-Schur, DY12, JS15, JSY16, JSY20}.

\subsection{A $\mathbb Z$-basis and a presentation for \( S_0(n,r) \)}
This subsection introduces a basis for $S_0(n, r)$, describes the multiplication rules among its elements, and presents a defining presentation of the algebra.

Beilinson, Lusztig, and MacPherson \cite{BLM90} provided a geometric construction of certain finite-dimensional quotients of the quantized enveloping algebra \( U_q(\mathfrak{gl}_n) \) using double flag varieties.  
Later, Du \cite{Du95} observed that these quotients are isomorphic to the \( q \)-Schur algebras \( S_q(n, \cdot) \), thereby giving a geometric realization of \( S_q(n,r) \).

In a similar manner, Jensen and Su \cite{JS15} provided a geometric realization of the \( 0 \)-Schur algebra \( S_0(n,r) \). Specifically, they constructed a \(\mathbb{Z}\)-algebra that is isomorphic to \( S_0(n,r) \), with a basis given by the set \( \mathcal{F} \times \mathcal{F}  / \mathrm{GL}(V) \) of \(\mathrm{GL}(V)\)-orbits in the double flag variety \( \mathcal{F} \times \mathcal{F} \). Here, \( V \) is an \( r \)-dimensional vector space over \( \mathbb{C} \), and \( \mathcal{F} \) is the variety of \( n \)-step flags in \( V \) over \( \mathbb{C} \).  
Since there exists a natural one-to-one correspondence between \( \mathcal{F} \times \mathcal{F} / \mathrm{GL}(V) \) and the set \( M_n(r) \) of \( n \times n \) matrices with nonnegative integer entries whose total sum of entries is \( r \), it follows that \( S_0(n, r) \) admits a basis indexed by elements of \( M_n(r) \), given by  
\begin{equation}\label{standard basis}
   \{ e_A \mid A \in M_n(r) \}. 
\end{equation}
The precise definition of \( e_A \) can be found in \cite[Section 1]{JS15}. In the present paper, this basis is referred to as the \emph{standard basis} for $S_0(n, r)$.

Next, we describe the multiplication rule for the elements of the standard basis.
For $A \in M_n(r)$, set
\begin{equation*}
\begin{aligned}
    \ro(A) &:= \left(\sum_{1\le j \le n} a_{1,j},\sum_{1\le j \le n} a_{2,j},\ldots,  \sum_{1\le j \le n} a_{n,j} \right) \\ 
    \co(A) &:= \left( \sum_{1\le j \le n} a_{i,1},\sum_{1\le j \le n} a_{i,2}, \ldots, \sum_{1\le j \le n} a_{i,n} \right).
\end{aligned}    
\end{equation*}
In other words, $\ro(A)$ and $\co(A)$ represent the row-sum vector and column-sum vector of the matrix $A$, respectively. 
It was shown in \cite[Corollary 7.2.2]{JS15} that the standard basis 
satisfies the following multiplication rule: for $A, B \in M_n(r)$, 
\begin{equation}\label{multiplication rule}
    e_A  e_B = 
    \begin{cases}
    e_C, & \text{if } \co(A) = \ro(B), \\
    0, & \text{otherwise},
    \end{cases}
 \end{equation}  
for a uniquely determined matrix \( C \in M_n(r) \).  
In some special cases, $C$ admits a simple description.
For each \( \lambda = (\lambda_1, \lambda_2, \ldots, \lambda_n) \in \Lambda(n, r) \), we write \( k_{\lambda} \) for \( e_{D_\lambda} \), where \( D_\lambda \) is the diagonal matrix  
$D_\lambda = \operatorname{diag}(\lambda_1, \lambda_2, \ldots, \lambda_n)$.  
By \cite[Eq. 4]{JSY20}, we have    
\begin{equation} 
    \label{eq: k_lambda act on e_A}
\begin{aligned}
    k_{\lambda}  e_A = 
    \begin{cases}
    e_A, &\text{if } \lambda = \ro(A)\\
    0, &\text{otherwise}
    \end{cases} \quad \text{ and } \quad 
   e_A  k_{\lambda} = 
    \begin{cases}
    e_A, &\text{if } \lambda = \co(A)\\
    0, &\text{otherwise.}
    \end{cases}
\end{aligned}
\end{equation}
It follows from \cref{eq: k_lambda act on e_A} that $$\sum_{\lambda \in \Lambda(n,r)} k_{\lambda}$$ is the identity element of $S_0(n,r)$.
For $1 \le i \le n-1$ and $\lambda \in \Lambda(n,r)$, set  
\begin{align*}
e_{i,\lambda}:=e_{D_{\lambda} - E_{i+1,i+1} + E_{i, i+1}} \quad \text{ and } \quad 
f_{i,\lambda}:=e_{D_{\lambda} - E_{i,i} + E_{i+1, i}},
\end{align*}
where $E_{a, b}$ denotes the $n\times n$ matrix with $1$ in the $(a,b)$-position and $0$s elsewhere for $1 \le a, b \le n$.
We then set 
\begin{align*}
e_i := \sum_{\lambda \in \Lambda(n,r)} e_{i, \lambda} \quad \text{ and } \quad 
f_i := \sum_{\lambda \in \Lambda(n,r)} f_{i, \lambda}.
\end{align*}
In view of \cref{multiplication rule}, for $A \in M_n(r)$, one has that
$$e_i  e_A = e_{i, \ro(A)} e_A \quad \text{ and } \quad f_i  e_A = f_{i, \ro(A)} e_A.$$ 
The following lemma presents the multiplication rule for $e_i e_A$ and $f_i e_A$.

\begin{lemma} {\rm (\cite[Lemma 6.11]{JS15})} \label{lem: multiplication rule}
Let $A \in M_n(r)$ with $\ro(A)=(\lambda_1, \lambda_2,\ldots, \lambda_n)$.
\begin{enumerate}[label = {\rm (\arabic*)}]
    \item If $\lambda_{i+1}>0$, then 
    $$e_i e_A = e_{A + E_{i,p} - E_{i+1,p}},$$ 
    where $p = \max\{ 1\le j \le n-1 \mid a_{i+1,j}>0 \}$.
    \item If $\lambda_{i}>0$, then 
    $$f_i e_A = e_{A - E_{i,q} + E_{i+1,q}},$$ 
    where $q = \min\{1\le j \le n-1 \mid a_{i,j}>0 \}$.
\end{enumerate}
\end{lemma}
It is worth noting that for \( A, B \in M_n(r) \), the explicit form of \( C \) in \cref{multiplication rule} can be determined by combining \cref{lem: multiplication rule} with \cite[Lemma 2.7]{JSY20}.  

Finally, we introduce the presentation of $S_0(n,r)$ as given in \cite{DY12, JS15}.
Set 
\begin{align*}
   \quad \epsilon_i &:=(0,\ldots,0,\stackrel{i{\rm th}}{1}, 0,0,\ldots,0)\in \mathbb{Z}^n, && (1\le i \le n),\\
   \quad a_i &:= \epsilon_i-\epsilon_{i+1},&&(1\le i \le n-1),\\
   \quad p_{i,j} &:= 
    \begin{cases}
        2a_i + a_j & \text{if } i=j \pm 1, \\
        a_i + a_j & \text{otherwise}
    \end{cases} && (1\le i,j \le n-1).
\end{align*}
With this notation, the desired presentation for 
$S_0(n,r)$ is given as follows.

\begin{theorem}{\rm (\cite{DY12, JS15})}
As a $\mathbb{Z}$-algebra, $S_0(n, r)$ is generated by $e_i$, $f_i$ ($1 \le i \le n-1$) and $k_{\lambda}$ ($\lambda \in \Lambda(n,r)$) with the following defining relations: for $1\le i,j \le n-1$ and $\lambda \in \Lambda(n,r)$,
\begin{align*}
   & k_{\lambda+p_{i,j}} P_{i,j} k_{\lambda}=0, \\
   & k_{\lambda-p_{i,j}} N_{i,j} k_{\lambda}=0, \\
   &  k_{\lambda + a_i - a_j} C_{i,j} k_{\lambda}=0.
\end{align*}
Here,  
\begin{align*}
  &  P_{i,j} = 
    \begin{cases}
        e_i^2 e_j - e_i e_j e_i & \text{for } i=j-1, \\
        - e_i e_j e_i + e_j e_i^2 & \text{for } i=j+1, \\
        e_i e_j - e_j e_i & \text{otherwise},
    \end{cases}\\
 &   N_{i,j} = 
    \begin{cases}
        - f_i f_j f_i + f_j f_i^2 & \text{for } i=j-1, \\
        f_i^2 f_j - f_i f_j f_i & \text{for } i=j+1, \\
        f_i f_j - f_j f_i & \text{otherwise},
    \end{cases}
\\
  &  C_{i,j} = e_i f_j - f_j e_i -\delta_{i,j} \left( \sum_{\mu: \mu_{i+1} =0} k_{\mu} - \sum_{\mu: \mu_{i} =0} k_{\mu}  \right).
\end{align*}
\end{theorem}

\section{Simple modules of the $0$-Schur algebra $\mathbf{S}_0(n, r)$}
\label{main result 1}

Let $r$ and $n$ be nonnegative integers.
Hereafter, all modules and algebras are defined over $\mathbb{C}$, and we consider the $\mathbb{C}$-algebra
$$
\mathbf{S}_0(n, r) := S_0(n, r) \otimes_{\mathbb{Z}} \mathbb{C}.
$$
Given an \( \mathbf{S}_0(n, r)\)-module \( M \), we denote its radical by \( \operatorname{rad}(M) \) and define the top of \( M \) as  
\[
\operatorname{top}(M) := M / \operatorname{rad}(M).
\]
There is a canonical bijection between the set of isomorphism classes of projective indecomposable $\mathbf{S}_0(n, r)$-modules and that of simple $\mathbf{S}_0(n, r)$-modules,
given by
$$
P \mapsto \operatorname{top}(P)$$
(for example, see \cite{15TO} or  \cite[Section 1]{95ARS}).

The purpose of this section is to provide a direct construction of simple modules for $\mathbf{S}_0(n, r)$.
As a consequence, we obtain an explicit description of the radicals of projective indecomposable $\mathbf{S}_0(n, r)$-modules.
We begin by recalling the construction of projective indecomposable $\mathbf{S}_0(n, r)$-modules as given in \cite{JSY16}.

\subsection{Projective indecomposable $\mathbf{S}_0(n, r)$-modules}

Let
$$
\lambda = (\lambda_1, \lambda_2, \ldots, \lambda_n) \in \Lambda(n, r), \quad  
\alpha = (\alpha_1, \alpha_2, \ldots, \alpha_s) \in \Lambda^+(n).
$$
Then $\lambda$ can be decomposed into subsequences as follows:
$$
(\lambda_1, \ldots, \lambda_{\alpha_1}),\ 
(\lambda_{\alpha_1+1}, \ldots, \lambda_{\alpha_1+\alpha_2}),\ 
\ldots,\ 
(\lambda_{\sum_{1 \le i \le s-1} \alpha_i + 1}, \ldots, \lambda_n).
$$
If each subsequence in this decomposition is either equal to $(0)$, or has neither its first nor last entry equal to zero, then $\alpha$ is said to be \emph{maximal with respect to $\lambda$}.
Let $\max(\lambda)$ denote the set of such strong compositions.
For example, if $\lambda = (3,0,2,0,1) \in \Lambda(5,6)$, then
$$
\max(\lambda) = \{ (5),\ (3,1,1),\ (1,1,3),\ (1,1,1,1,1) \}.
$$
This set admits the following simple description.
For $\beta = (\beta_1, \beta_2, \ldots, \beta_k) \in \Lambda^+(l)$ with $l \le n$, let $\beta \cdot 1^{n - l}$ denote the concatenation of $\beta$ with $1^{n - l}$, that is,
$$
\beta \cdot 1^{n - l} = (\beta_1, \beta_2, \ldots, \beta_k, \underbrace{1, \ldots, 1}_{n - l\text{ times}}).
$$
Then
\begin{equation}\label{simple description for max(lambda)}
\max(\lambda) = \{ \beta \cdot 1^{n - \ell(\lambda^+)} \mid \beta \in \Lambda^+(\ell(\lambda^+)) \}.
\end{equation}

Similarly, we can decompose an \( n \times n \) matrix \( A \) into \( s \) \( n \times n \) matrices as follows:
$$A(\alpha;1),\,  A(\alpha;2),\, \ldots \,, \,A(\alpha;s),$$
where, for $1\le v\le s$,  
\begin{align*}
A(\alpha;v)^{(i)}=
\begin{cases}
A^{(i)}& \text{ if } \sum_{t=1}^{v-1}\alpha_t+1 \le i \le \sum_{t=1}^{v}\alpha_t,\\
0 & \text{otherwise.}
\end{cases}  
\end{align*}
Here, \( B^{(i)} \) denotes the \( i \)th column of \( B \) for any square matrix \( B \).
An $n \times n$ matrix $A$ is called \emph{open} if every $2 \times 2$ submatrix of $A$ has at least one zero on its diagonal. Additionally, $A$ is called \emph{open on columns} with respect to $\alpha$ if $A(\alpha; v)$ is an open matrix for all $v = 1, \ldots, s$.

Let  
    \begin{equation*}
    \mathcal{B}^{\lambda,\alpha} :=\{ e_A \mid \text{$\mathrm{co}(A)=\lambda$ and $A$ is open on columns with respect to $\alpha$}\} 
    \end{equation*}
    and let 
    \begin{equation*}
     B^{\lambda} := \left\{e_A \mid \co(A) = \lambda \right\} \,\,\setminus \,\,\bigcup_{ \alpha \in \rm{max}(\lambda) \setminus \{(1,1,\ldots,1) \} } B^{\lambda, \alpha}.
    \end{equation*}
    From \cref{simple description for max(lambda)} it follows that 
     \begin{equation}\label{basis for Pims}
        B^{\lambda} 
        = \{e_A \mid \co(A) = \lambda \} \,\,\setminus \,\,\bigcup_{ \alpha \in \{ \beta \cdot 1^{n-\ell(\lambda^+)}  | \beta \in \Lambda^+(\ell(\lambda^+))  \} \setminus \{1^n\} } B^{\lambda, \alpha}.
    \end{equation}
Following \cite{JSY16}, we define an $\mathbf{S}_0(n, r)$-action on the $\mathbb{C}$-span of $B^{\lambda}$ as follows: 
for $A \in B^{\lambda}$ and $b \in \{e_i, f_i, k_{\lambda} \}$, 
     \begin{equation}\label{action of ei and fi}
        b \cdot e_A := 
        \begin{cases}
            e_B    &   \text{ if $b \cdot e_A = e_B$ in $S_0(n,r)$ and $e_B \in B^{\lambda}$}, \\ 
            0    &    \text{ otherwise.}
        \end{cases}
    \end{equation}
    This indeed defines an $\mathbf{S}_0(n, r)$-module, denoted by $P_{\lambda}$. 
   
    \begin{theorem} {\rm (\cite[Theorem 5.1]{JSY16})}
For each $\lambda \in \Lambda(n, r)$, the module $P_\lambda$ is a projective indecomposable $\mathbf{S}_0(n, r)$-module.
    \end{theorem}

    Recall that \( \Lambda^{\bullet}(n,r) \) is the set of compositions \( (\lambda_1, \lambda_2, \dots, \lambda_n) \in \Lambda(n,r) \) such that if \( \lambda_i = 0 \) for some \( i \), then \( \lambda_j = 0 \) for all \( j > i \) 
    (see \cref{relation between weak and strong compositions}). 
    Indeed, this set serves as a complete set of representatives for the equivalence classes of $\Lambda(n,r)/\sim$, where the equivalence relation $\sim$ is defined in \cite[Section 5.4]{JSY16}.

    \begin{theorem} {\rm (\cite[Theorem 5.6]{JSY16})}\label{thm: pim complete set}
    The set $\{ P_{\lambda} \mid \lambda \in \Lambda^{\bullet}(n,r) \}$ forms a complete set of representatives for the isomorphism classes of projective indecomposable  $\mathbf{S}_0(n, r)$-modules.
    \end{theorem}

\begin{example}
        Let $\lambda = (2,1,0) \in \Lambda^{\bullet}(3,3)$.
        Then
        $$\{ \beta \cdot 1^{n-\ell(\lambda^+)} \mid \beta \in \Lambda^{+}(\ell(\lambda^+)) \}=\{ \beta \cdot 1^1 \mid \beta \in \Lambda^{+}(2) \} = \{(2,1), (1,1,1) \} .$$  
        It follows that \( B^{\lambda} \) consists of the basis elements \( e_A \), where the matrices $A$ are given by   
        \[
            \begin{bmatrix}
                2 & 0 & 0 \\
                0 & 1 & 0 \\
                0 & 0 & 0
                \end{bmatrix},
                \quad
                \begin{bmatrix}
                2 & 0 & 0 \\
                0 & 0 & 0 \\
                0 & 1 & 0
                \end{bmatrix},
                \quad
                \begin{bmatrix}
                1 & 0 & 0 \\
                1 & 1 & 0 \\
                0 & 0 & 0
                \end{bmatrix},
                \quad
                \begin{bmatrix}
                1 & 0 & 0 \\
                1 & 0 & 0 \\
                0 & 1 & 0
                \end{bmatrix},
        \]
        \[
            \begin{bmatrix}
                1 & 0 & 0 \\
                0 & 1 & 0 \\
                1 & 0 & 0
                \end{bmatrix},
                \quad
            \begin{bmatrix}
                0 & 0 & 0 \\
                2 & 0 & 0 \\
                0 & 1 & 0
                \end{bmatrix},
                \quad
            \begin{bmatrix}
                1 & 0 & 0 \\
                0 & 0 & 0 \\
                1 & 1 & 0
                \end{bmatrix},
                \quad
                \begin{bmatrix}
                0 & 0 & 0 \\
                1 & 0 & 0 \\
                1 & 1 & 0
                \end{bmatrix}.
        \]
        The actions of $e_i, f_i$ on $e_A$ for each $A \in B^{\lambda}$ are illustrated in \cref{fig:Structure of P_210}.
            In the figure, each node labeled by a matrix $A$ represents the element $e_A \in \beta^{\lambda}$. If there is a directed edge labeled by $b \in \{e_i, f_i\}$ from $A$ to $B$, it indicates that $b \cdot e_A = e_B$. Moreover, if for a node labeled by $A$ there is no directed edge labeled by $b$ with source $A$, it signifies that $b \cdot e_A = 0$. The action of $k_{\mu} \in \Lambda(3,3)$ is the same as in \cref{eq: k_lambda act on e_A}. 
    \end{example}
\begin{figure}[h!]
            \centering
\scalebox{0.7}{\begin{tikzpicture}[
  matrix/.style={text width=1cm,align=center},
  every node/.style={anchor=center},
  e_arrow/.style={->, draw=blue, thick},
  f_arrow/.style={->, draw=red, thick}
]

\node[circle,draw] (A) [matrix] {
  200\\010\\000
};

\node[circle,draw] (B) [matrix, below left=1cm and 1cm of A] {
  100 \\
  110 \\
  000 \\
};

\node[circle,draw] (C) [matrix, below right=1cm and 1cm of A] {
  200 \\
  000 \\
  010 \\
};

\node[circle,draw] (D) [matrix, below=1.4cm of B] {
  100 \\
  010 \\
  100 \\
};

\node[circle,draw] (E) [matrix, below=1.4cm of C] {
  100 \\
  100 \\
  010 \\
};

\node[circle,draw] (F) [matrix, below=1.4cm of D] {
  100 \\
  000 \\
  110 \\
};

\node[circle,draw] (G) [matrix, below=1.4cm of E] {
  000 \\
  200 \\
  010 \\
};

\node[circle,draw] (H) [matrix, below left=1cm and 1cm of G] {
  000 \\
  100 \\
  110 \\
};

\draw[f_arrow] (A) -- (B) node[midway, left] {$f_1$};

\draw[f_arrow] (E) -- (F) node[midway, left] {$f_2$};

\draw[f_arrow] (G) -- (H) node[midway, right] {$f_2$};

\draw[e_arrow] (E) -- (B) node[midway, left] {$e_2$};

\draw[f_arrow, bend left] (A) to node[midway, right] {$f_2$} (C);
\draw[e_arrow, bend left] (C) to node[midway, left] {$e_2$} (A);

\draw[f_arrow, bend left] (C) to node[midway, right] {$f_1$} (E);
\draw[e_arrow, bend left] (E) to node[midway, left] {$e_1$} (C);

\draw[f_arrow, bend left] (E) to node[midway, right] {$f_1$} (G);
\draw[e_arrow, bend left] (G) to node[midway, left] {$e_1$} (E);

\draw[f_arrow, bend left] (B) to node[midway, right] {$f_2$} (D);
\draw[e_arrow, bend left] (D) to node[midway, left] {$e_2$} (B);

\draw[f_arrow, bend left] (D) to node[midway, right] {$f_2$} (F);
\draw[e_arrow, bend left] (F) to node[midway, left] {$e_1$} (D);

\draw[f_arrow, bend left] (F) to node[midway, right] {$f_1$} (H);
\draw[e_arrow, bend left] (H) to node[midway, left] {$e_1$} (F);

\end{tikzpicture}}
\caption{The $S_0(3,3)_{\mathbb{C}}$-action on $P_{(2,1,0)}$}
\label{fig:Structure of P_210}      
 \end{figure}
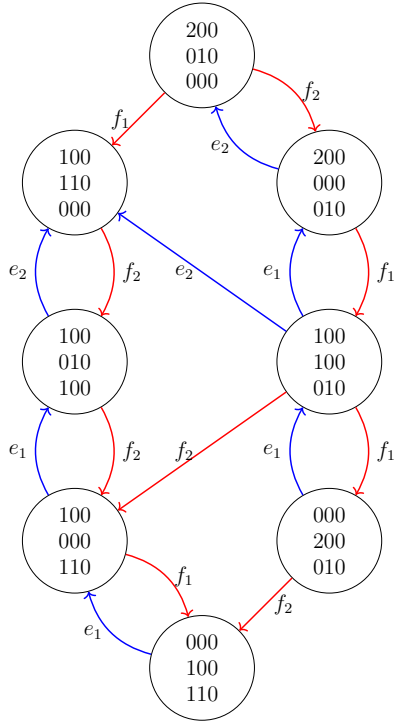

\subsection{Construction of simple $\mathbf{S}_0(n, r)$-modules}

This bijection is given by  
$P_\lambda \mapsto \operatorname{top}(P_\lambda)$  
for all \( \lambda \in \Lambda^{\bullet}(n,r) \).  
However, to the best of the authors' knowledge, an explicit description of \( \operatorname{rad}(P_\lambda) \) has not yet been provided.

In this section, we present a direct construction of simple modules for $\mathbf{S}_0(n, r)$.  
Using this construction, we also describe the radicals of projective indecomposable $\mathbf{S}_0(n, r)$-modules.

\begin{definition}\label{def: column block diagonal matrix}
    A \emph{column block diagonal matrix} is a finite matrix $A = (a_{i,j})$ with nonnegative integer entries satisfying the condition that 
    if $a_{i,j} > 0$, then $a_{r,s} = 0$ for all $r \le i$ and $s > j$.
    \end{definition}
    
    For $\lambda \in \Lambda^{\bullet}(n,r)$, let $\cb(\lambda)$ denote the set of $n \times n$ column block diagonal matrices $A$ satisfying $\co(A) = \lambda$.
    
    \begin{example}\label{ex: CB}
Let $\lambda = (2,1,0) \in \Lambda^\bullet(3,3)$. The set $\cb(\lambda)$ consists of the following matrices:
\[
\begin{bmatrix}
2 & 0 & 0 \\
0 & 1 & 0 \\
0 & 0 & 0
\end{bmatrix}, \quad
\begin{bmatrix}
2 & 0 & 0 \\
0 & 0 & 0 \\
0 & 1 & 0
\end{bmatrix}, \quad
\begin{bmatrix}
1 & 0 & 0 \\
1 & 0 & 0 \\
0 & 1 & 0
\end{bmatrix}, \quad
\begin{bmatrix}
0 & 0 & 0 \\
2 & 0 & 0 \\
0 & 1 & 0
\end{bmatrix}.
\]
    \end{example}

    Let $\lambda = (\lambda_1,\ldots,\lambda_t,0,\ldots,0)$ where $\lambda_i>0$ for $1 \le i \le t$ and $A \in \cb(\lambda)$. Then the following conditions hold by construction.
\begin{enumerate}
    \item For each $1 \le j \in t$, the $j$th column contains at least one nonzero entry.
    \item For each $1 \le i \le n $, the $i$th row contains at most one nonzero entry.
    \item For each $1 \le i \le n$, if $a_{i,j} > 0$ and $a_{i+1,k} > 0$, then $j \le k$.
\end{enumerate}
Set
\begin{align}\label{eq: m_j}
    m_j := 
    \begin{cases}
        1    &    \mbox{if $j=1$},\\
        \text{the smallest index $i$ such that } a_{i,j} > 0    &   \mbox{if $2 \le j \le t$}, \\ 
        n+1    &    \mbox{if $j=t+1$}.
    \end{cases}    
\end{align}
For each $1 \le j \le t$, we define the $j$th \emph{column block} of $A$, denoted $B_j(A)$, to be the submatrix of $A$ consisting of the rows from $m_j$ to $m_{j+1} - 1$ in column $j$. Then $A$ takes the form
\begin{align}\label{eq: column block of A}
A = 
\left(
\begin{array}{cccc}
B_1(A) & \mathbf{0} & \cdots & \mathbf{0} \\
\mathbf{0} & B_2(A) & \cdots & \mathbf{0} \\
\vdots & \vdots & \ddots & \vdots \\
\mathbf{0} & \mathbf{0} & \cdots & B_t(A)
\end{array}
\;\middle|\;
\mathbf{0}
\right),    
\end{align}
which justifies the terminology \emph{column block diagonal matrix}.

We now define
\[
\beta^{\lambda} := \{ e_A \mid A \in \cb(\lambda) \}.
\]
Since the diagonal matrix $\operatorname{diag}(\lambda_1, \lambda_2, \ldots, \lambda_n)$ lies in $\cb(\lambda)$, the basis element $k_\lambda$ is contained in $\beta^\lambda$. Moreover, as $\beta^{\lambda} \subseteq B^{\lambda,(1,1,\ldots,1)}$, it follows that $\beta^{\lambda} \subseteq B^{\lambda}$.

\begin{proposition}\label{submodule of PIM}
For each $\lambda \in \Lambda^{\bullet}(n,r)$, the $\mathbb{C}$-span of $B^{\lambda} \setminus \beta^{\lambda}$ forms an $\mathbf{S}_0(n, r)$-submodule of $P_\lambda$.
    \end{proposition}
    
    \begin{proof}    
        Let $e_A \in B^{\lambda} \setminus \beta^{\lambda}$.  
        For the assertion, we only to see that $b \cdot e_A \in (B^{\lambda} \setminus \beta^{\lambda}) \cup \{ 0 \}$ for $b \in \{ e_i, f_i \mid 1 \le i \le n - 1 \} \cup \{ k_\mu \mid \mu \in \Lambda(n,r) \}$.

        For the sake of contradiction, assume that $e_i \cdot e_A \in \beta^{\lambda}$. 
        Then every row of $A$, except for the $i$th and $(i+1)$th rows, can have at most one nonzero entry. 
Recall that acting with $e_i$ on the left decreases the rightmost nonzero entry of the $(i+1)$th row by 1 and increases the entry of the $i$th row in the same column by 1 (see \cref{lem: multiplication rule} and \cref{action of ei and fi}). 
This implies that the $(i+1)$th row of $A$ can also have at most one nonzero entry. If it has two or more nonzero entries, then the action of $e_i$ will result in $e_i \cdot e_A \notin \beta^{\lambda}$. Therefore the following two cases are possible.
        
        \smallskip
        \noindent
        \emph{Case 1.} The $(i+1)$st row of $A$ is zero.  
        Then $e_i \cdot e_A = 0$ by definition.
        
        \smallskip
        \noindent
        \emph{Case 2.} The $(i+1)$st row has a unique nonzero entry at position $(i+1,s)$. Consider the $i$th row of $A$. Let $B \in \cb(\lambda)$ such that $e_i \cdot e_A = e_B$.
        
        \begin{itemize}
            \item Suppose that the $i$th row of $A$ is zero. Then $B$ has a unique nonzero entry in each of the $i$th and $(i+1)$st rows in the same column. This implies $e_A = f_i \cdot e_B$ and $A \in \cb(\lambda)$, contradicting the assumption.
        
            \item Suppose that the $i$th row of $A$ contains a nonzero entry,  say at position $(i,t)$. 
            If $t \neq s$, then $B$ has nonzero entries in row $i+1$ at $(i+1,t)$ and $(i+1,s)$. But this cannot occur since $B \in \cb(\lambda)$.
            This tells us that $t=s$ and both the $i$th and $(i+1)$st rows of $A$ have unique nonzero entries in the same column. From this it follows that $A \in \cb(\lambda)$, which contradicts the assumption that $A \notin \cb(\lambda)$.
        \end{itemize}
        
        Therefore, $e_i \cdot e_A \in B^{\lambda} \setminus \beta^{\lambda} \cup \{ 0 \}$ in all cases. 
        
        A similar argument shows that $f_i \cdot e_A$ also lies in the same set. And, from
        \[
        k_{\mu} \cdot e_A =
        \begin{cases}
        e_A    & \text{if } \operatorname{row}(A) = \mu, \\ 
        0      & \text{otherwise},
        \end{cases}
        \]
        it follows that 
        $k_{\mu} \cdot e_A \in (B^{\lambda} \setminus \beta^{\lambda}) \cup \{ 0 \}$.
    \end{proof}

For $\lambda \in \Lambda^{\bullet}(n,r)$, let $N_\lambda$ denote the submodule $\mathbb{C}[B^{\lambda} \setminus \beta^{\lambda}]$ of $P_\lambda$.
Let $$S_\lambda := P_\lambda / N_{\lambda}.$$ 
Then  
\[
\left\{ \overline{e_A} := e_A + N_{\lambda} \mid A \in \cb(\lambda) \right\}
\]
forms a basis for $S_\lambda$.
The action of $S_0(n,r)$ on this basis is given by
\[
b \cdot \overline{e_A} =
\begin{cases}
\overline{e_B} & \text{if } b \cdot e_A = e_B \text{ and } e_B \in \beta^{\lambda}, \\
0 & \text{otherwise},
\end{cases}
\]
for $b \in \{ e_i, f_i \mid 1 \le i \le n - 1 \} \cup \{ k_\mu \mid \mu \in \Lambda(n,r) \}$.

\begin{example}
    Let $\lambda = (2,1,0) \in \Lambda^\bullet(3,3)$. 
    The actions of $e_i$ and $f_i$ on the basis $\{ \overline{e_A} \mid A \in \cb(\lambda) \}$ are illustrated in \cref{fig:Structure of S_210}.
    \end{example}

\begin{figure}[h!]
    \centering

    \scalebox{0.7}{\begin{tikzpicture}
    [
        matrix/.style={text width=1cm,align=center},
        every node/.style={anchor=center},
        e_arrow/.style={->, draw=blue, thick},
        f_arrow/.style={->, draw=red, thick}
      ]
      
      \node[circle,draw] (A) [matrix] {
        200\\010\\000
      };

      \node[circle,draw] (C) [matrix, below=1.4cm of A] {
        200 \\
        000 \\
        010 \\
      };

      \node[circle,draw] (E) [matrix, below=1.4cm of C] {
        100 \\
        100 \\
        010 \\
      };

      \node[circle,draw] (G) [matrix, below=1.4cm of E] {
        000 \\
        200 \\
        010 \\
      };


      \draw[f_arrow, bend left] (A) to node[midway, right] {$f_2$} (C);
      \draw[e_arrow, bend left] (C) to node[midway, left] {$e_2$} (A);

      \draw[f_arrow, bend left] (C) to node[midway, right] {$f_1$} (E);
      \draw[e_arrow, bend left] (E) to node[midway, left] {$e_1$} (C);
      
      \draw[f_arrow, bend left] (E) to node[midway, right] {$f_1$} (G);
      \draw[e_arrow, bend left] (G) to node[midway, left] {$e_1$} (E);
      
      \end{tikzpicture}
      }

      \caption{The $S_0(3,3)_{\mathbb{C}}$-action on $S_{(2,1,0)}$}
      \label{fig:Structure of S_210}      
\end{figure}
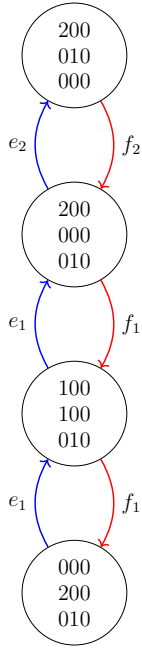

Given two weak compositions $\mu = (\mu_1, \ldots, \mu_n)$ and $\lambda = (\lambda_1, \ldots, \lambda_n)$ of $r$, we say that $\mu$ is a \emph{refinement} of $\lambda$, written $\mu \preceq \lambda$, if $\mu$ can be obtained by subdividing the parts of $\lambda$ in order.
Let
\begin{align}\label{eq: refinement}
\Lambda(n,r)_{\preceq \lambda}:= \{ \mu \in \Lambda(n,r) \mid \mu \preceq \lambda \}.
\end{align}
\begin{proposition}\label{prop: row bijection}
For each $\lambda \in \Lambda^{\bullet}(n,r)$, the map 
\[
\ro : \cb(\lambda) \to \Lambda(n,r)_{\preceq \lambda}, \quad A \mapsto \ro(A)
\]
is a bijection. Here, $\ro(A)$ is the row-sum vector of $A$.
\end{proposition}

\begin{proof}
Let $A \in \cb(\lambda)$ and let $\{ B_j(A) \}$ denote the column blocks of $A$. Then the row-sum vector of $A$ is given by
\[
\ro(A) = 
\left(
\begin{array}{c}
B_1(A)\\
B_2(A)\\
\vdots\\
B_t(A)
\end{array}
\right)^{\mathsf{T}}.
\]

To show that $\ro$ is a bijection, we construct its inverse. Let $\mu = (\mu_1, \ldots, \mu_n) \in \Lambda(n,r)_{\preceq \lambda}$. Since $\mu \preceq \lambda$, there exists a unique sequence of indices
\[
1 = i_1 < i_2 < \cdots < i_{t+1} = n+1
\]
such that
\[
\lambda_k = \mu_{i_k} + \mu_{i_k+1} + \cdots + \mu_{i_{k+1}-1}
\quad \text{for each } 1 \le k \le t,
\]
and $\mu_{i_k} > 0$ for all $2 \le k \le t$.

Define the matrix $A_\mu = (b_{i,j})$ by
\[
b_{i,j} := 
\begin{cases}
\mu_i & \text{if } i_j \le i < i_{j+1}, \\
0     & \text{otherwise}.
\end{cases}
\]
Then $A_\mu \in \cb(\lambda)$ satisfies $\ro(A_\mu) = \mu$, and for any $A \in \cb(\lambda)$, we have $A_{\ro(A)} = A$. Hence, the map $\mu \mapsto A_\mu$ defines the inverse of $\ro$.
\end{proof}

We introduce several remarkable properties of the basis $\beta^\lambda$ of $S_\lambda$.

\begin{lemma}\label{lem: for main thm}
Let $\lambda \in \Lambda^{\bullet}(n,r)$, and let $A \in \cb(\lambda)$. 
\begin{enumerate}[label = {\rm (\arabic*)}]
    \item
    There exist indices $1 \leq i_1, i_2, \ldots, i_l \leq n-1$ such that 
    \[
    e_{i_1} e_{i_2} \cdots e_{i_l} \cdot \overline{e_A} = \overline{k_\lambda}.
    \]
    
    \item
    The identity in {\rm(1)} holds if and only if 
    \[
    \overline{e_A} = f_{i_l} \cdots f_{i_2} f_{i_1} \cdot \overline{k_\lambda}.
    \]
    
    \item
    Suppose
    \[
    \overline{e_A} = f_{i_l} \cdots f_{i_2} f_{i_1} \cdot \overline{k_\lambda} = f_{j_s} \cdots f_{j_2} f_{j_1} \cdot \overline{k_\lambda}.
    \]
    Then the multisets $\{i_1, \ldots, i_l\}$ and $\{j_1, \ldots, j_s\}$ are equal. In particular, $l = s$.
\end{enumerate}
\end{lemma}

\begin{proof}
(1) Suppose that $A$ has column blocks as defined in~\cref{eq: column block of A}. Since the action of $e_i$ transfers a unit from row $i+1$ to row $i$ within the same column, one can successively apply suitable $e_i$'s to move all nonzero entries in each column block upward. In this way, we obtain a matrix $B$ with the same indices $m_j$ as in~\cref{eq: m_j}, and whose row-sum vector is given by
\[
\ro(B) = (\lambda_1, 0, \ldots, 0,\ \lambda_2, 0, \ldots, 0,\ \ldots,\ \lambda_t, 0, \ldots, 0).
\]
By further applying appropriate $e_i$'s to $\overline{e_B}$, we obtain $\overline{k_\lambda}$, which proves the assertion.

(2) Suppose $f_i \cdot \overline{e_M} = \overline{e_N}$ for some $M, N \in \cb(\lambda)$. By the multiplication rule, this implies that the $i$th row of $M$ and the $(i+1)$st row of $N$ each have a unique nonzero entry in the same column $j$. The operator $f_i$ decreases the $(i,j)$-entry and increases the $(i+1,j)$-entry, while $e_i$ reverses this transformation. Thus,
\[
e_i \cdot \overline{e_N} = \overline{e_M}.
\]
The converse follows by symmetry. This completes the proof of (2).

(3) Whenever $f_i \cdot \overline{e_M} = \overline{e_N}$, the row-sum vector changes by
\begin{align}\label{eq: row vector f action}
\ro(N) = \ro(M) + (\underbrace{0, \ldots, 0}_{i-1}, -1, 1, 0, \ldots, 0).
\end{align}
Therefore, the number of times $f_i$ appears in the expression $f_{i_l} \cdots f_{i_1} \cdot \overline{k_\lambda}$ is equal to the coefficient of $(\underbrace{0, \ldots, 0}_{i-1}, -1, 1, 0, \ldots, 0)$ in the difference of the vectors $\lambda - \ro(A)$. This determines the multiset $\{i_1, \ldots, i_l\}$ uniquely.
\end{proof}

    By \cref{lem: for main thm}~(1),~(2), every element $\overline{e_A} \in \beta^\lambda$ can be written as
\[
\overline{e_A} = f_{i_l} \cdots f_{i_2} f_{i_1} \cdot \overline{k_\lambda}
\]
for some indices $1 \le i_1, \ldots, i_l \le n - 1$.  
By \cref{lem: for main thm}~(3), the length $l$ is uniquely determined; we denote it by $\ell(\overline{e_A}) := l$.

The following lemma will be used in the proof of the main theorem in this section.

\begin{lemma}\label{lem: for main thm2}
    Let $\lambda \in \Lambda^{\bullet}(n,r)$ and $A, B \in \cb(\lambda)$. If
    \[
    \overline{e_A} = f_{i_l}  \cdots f_{i_2} f_{i_1} \cdot \overline{k_\lambda} \quad \text{ and }\quad \ell(\overline{e_B}) \le l,
    \]
    then
    \[
    e_{i_1} e_{i_2} \cdots  e_{i_l} \cdot \overline{e_B} =
    \begin{cases}
    \overline{k_\lambda} & \text{if } B = A, \\
    0 & \text{otherwise}.
    \end{cases}
    \]
    \end{lemma}

    \begin{proof}
        Suppose $\overline{e_A} = f_{i_l}  \cdots f_{i_2} f_{i_1} \cdot \overline{k_\lambda}$ and let $B \in \cb(\lambda)$.
        
        First, observe that for any $N \in \cb(\lambda)$ and any $1 \le i \le n-1$, we have
        \[
        \ell(f_i \cdot \overline{e_N}) = \ell(\overline{e_N}) - 1 \quad \text{unless } f_i \cdot \overline{e_N} = 0.
        \]
        In particular, since $e_i \cdot \overline{k_\lambda} = 0$ for all $i$, it follows that
        \[
            e_{i_1} e_{i_2} \cdots  e_{i_l} \cdot \overline{e_B} = 0 \quad \text{if } \ell(\overline{e_B}) < l.
        \]
        
        Now assume $\ell(\overline{e_B}) = \ell(\overline{e_A}) = l$. Then applying $e_{i_1} e_{i_2} \cdots  e_{i_l}$ to $\overline{e_B}$ is either $\overline{k_\lambda}$ or $0$, since $\overline{k_\lambda}$ is the unique element in $\beta^{\lambda}$ of length $0$.

        If $e_{i_1} e_{i_2} \cdots e_{i_l} \cdot \overline{e_B} = \overline{k_\lambda}$,
        then by \cref{lem: for main thm}~(2), we have
        \[
        \overline{e_B} = f_{i_l} \cdots f_{i_2} f_{i_1} \cdot \overline{k_\lambda} = \overline{e_A}.
        \]
        Hence $A = B$.
        \end{proof}
        
            We are now ready to state the main result of this section.
\begin{theorem}\label{thm: Slambda is simple}
The set $\{ S_{\lambda} \mid \lambda \in \Lambda^{\bullet}(n,r) \}$ forms a complete set of representatives of the isomorphism classes of simple $\mathbf{S}_0(n, r)$-modules.
\end{theorem}

\begin{proof}
Let $\lambda \in \Lambda^{\bullet}(n,r)$, and let $c \in S_\lambda$ be a nonzero element. Denote by $C$ the ${S_0(n,r)}_\mathbb{C}$-submodule of $S_\lambda$ generated by $c$. It suffices to show that $C = S_\lambda$.

Write $c = c_1 v_1 + c_2 v_2 + \cdots + c_k v_k$ for some $k \in \mathbb{Z}_{>0}$, where each $c_i \in \mathbb{C} \setminus \{0\}$ and $v_i \in \beta^\lambda$. Without loss of generality, assume that $v_1$ has maximal length among $\{ \ell(v_i) \mid 1 \le i \le k \}$.

Suppose $v_1 = f_{i_l} \cdots f_{i_2} f_{i_1} \cdot \overline{k_\lambda}$. Then, by \cref{lem: for main thm2}, we have
\[
e_{i_1} e_{i_2} \cdots e_{i_l} \cdot c = c_1 \overline{k_\lambda}.
\]
Hence $\overline{k_\lambda} \in C$. Since $S_\lambda$ is generated by $\overline{k_\lambda}$, it follows that $C = S_\lambda$. Thus $S_\lambda$ is simple.

Finally, since each projective indecomposable module $P_\lambda$ has a unique simple quotient, the assertion follows from \cref{thm: pim complete set}.
\end{proof}

\cref{thm: Slambda is simple} tells us that  $S_\lambda = {\rm top}(P_\lambda)$. This immediately yields the following corollary.

\begin{corollary}\label{Describing radical }
For each $\lambda \in \Lambda^{\bullet}(n,r)$, the submodule $N_{\lambda}$ of $P_\lambda$ is the radical of $P_\lambda$.
\end{corollary}

\section{Functorial Relations Among Module Categories}
\label{main result: 2}

As before, we fix nonnegative integers \( r \) and \( n \).
For a $\mathbb C$-algebra $A$, let $\text{$A$-\textsf{mod}}$ denote the category of finite-dimensional left $A$-modules.
In this section, we investigate 
functorial relations among three module categories 
$U_0(\mathfrak{gl}_n)$-\textsf{mod},  
$\mathbf{S}_0(n,r)$-\textsf{mod}, and $\mathbf{H}_r(0)$-\textsf{mod}.

\subsection{The $0$-Hecke algebra $\mathbf{H}_r(0)$ and its representation theory}\label{subsec: 0-Hecke alg and QSym}
The $0$-Hecke algebra $\mathbf{H}_r(0)$ is the associative $\C$-algebra with $1$ generated by the elements $\opi_1,\opi_2,\ldots,\opi_{r-1}$ subject to the following relations:
\begin{align}
\begin{aligned}\label{Rel: 0-Hecke algebra}
\opi_i^2 &= -\opi_i \quad \text{for $1\le i \le r-1$},\\
\opi_i \opi_{i+1} \opi_i &= \opi_{i+1} \opi_i \opi_{i+1}  \quad \text{for $1\le i \le r-2$},\\
\opi_i \opi_j &= \opi_j \opi_i \quad \text{if $|i-j| \ge 2$}.
\end{aligned}
\end{align}
Another set of generators consists of $\pi_i:= \opi_i + 1$ for $i=1,2,\ldots,r-1$ with the same relations as above except that $\pi_i^2 = \pi_i$.

For any reduced expression $s_{i_1} s_{i_2} \cdots s_{i_p}$ for $\sigma \in \SG_n$, let $\opi_{\sigma} := \opi_{i_1} \opi_{i_2} \cdots \opi_{i_p}$ and $\pi_{\sigma} := \pi_{i_1} \pi_{i_2 } \cdots \pi_{i_p}$.
It is well known that these elements are independent of the choices of reduced expressions, and both $\{\opi_\sigma \mid \sigma \in \SG_r\}$ and $\{\pi_\sigma \mid \sigma \in \SG_r\}$ are $\mathbb C$-bases for $\mathbf{H}_r(0)$.

According to \cite{79Norton}, there are $2^{r-1}$ pairwise inequivalent simple $\mathbf{H}_r(0)$-modules and $2^{r-1}$ pairwise inequivalent projective indecomposable $\mathbf{H}_r(0)$-modules, which are naturally indexed by strong compositions of $r$.
Let us explain this in more detail.
For $S \subseteq [n-1]$, let $\SG_{S}$ be the parabolic subgroup of $\SG_n$ generated by $\{s_i \mid i\in S\}$ and $w_0(S)$ the longest element in $\SG_{S}$. When $S=[n-1]$, we simply write $w_0$ for $w_0(S)$.
An element $w \in \SG_n$ can be written uniquely as $w=zu$, where $z\in \SG^S$ and $u \in \SG_S$, with the property that  $\ell(w)=\ell(z)+\ell(u)$.
Here $\SG^S:=\{z\in \SG_n \mid \operatorname{Des}(z) \subseteq S^{\rm c}\}$ is the set of minimal length representatives for left $\SG_S$-cosets, where $S^{\rm c}=[n-1]\setminus S$. 

\begin{lemma} {\rm (\cite[Theorem 6.2]{88BW})} \label{Bjorner and Wachs}
Given $S \subseteq T \subseteq [n-1]$,
the set $\{w\in \SG_n \mid S \subseteq \operatorname{Des}(w) \subseteq T\}$ is exactly the weak Bruhat interval $[w_0(S), w_1(T)]_L$,
where $w_0(S)$ is the longest element in $\SG_{S}$ and $w_1(T)$ is the longest element in $\SG^{{T}^{\rm c}}$.
\end{lemma}

Let $\alpha$ be a strong composition of $r$. The simple $\mathbf{H}_r(0)$-module
corresponding to $\alpha$ 
is given by $F_{\alpha}:=\mathbf{H}_r(0) \opi_{w_1(\set(\alpha)^{\rmr})}\pi_{w_0(\set(\alpha)^\rmt)}$. 
This module is one-dimensional, with basis element
$
v_\alpha := \opi_{w_1(\set(\alpha)^{\rm r})}\pi_{w_0(\set(\alpha)^{\rm t})}.
$
The action of $\mathbf{H}_r(0)$ on $v_\alpha$ is given, for each $i \in [r-1]$, by
$$
\opi_i(v_\alpha) =
\begin{cases}
- v_\alpha, & i \in \set(\alpha),\\
0, & i \notin \set(\alpha).
\end{cases}
$$
And, the projective indecomposable $\mathbf{H}_r(0)$-module corresponding to $\alpha$
is given by the submodule 
$R_\alpha := \mathbf{H}_r(0) \opi_{w_0(\set(\alpha))}\pi_{w_0(\set(\alpha)^\rmc)} $ of the regular representation of $\mathbf{H}_r(0)$.
Then it holds that 
$\operatorname{top}(R_\alpha) \cong F_\alpha.$

For later use, we introduce the two involutions $\autophi, \autotheta$ and the anti-involution $\autochi$ of $\mathbf{H}_r(0)$ introduced by Fayers in \cite{05Fayers}, defined as follows:
\begin{equation}\label{automorphism of Fayers}
\begin{aligned}
&\autophi: H_r(0) \ra H_r(0), \quad \pi_i \mapsto \pi_{n-i} \quad \text{for $1 \le i \le r-1$},\\
&\autotheta: H_r(0) \ra H_r(0), \quad \pi_i \mapsto - \opi_i \quad \text{for $1 \le i \le r-1$}, \\
&\autochi: H_r(0) \ra H_r(0), \quad \pi_i \mapsto \pi_i \quad \text{for $1 \le i \le r-1$}.
\end{aligned}
\end{equation}
These morphisms commute with each other.

In the rest of this subsection, we explain the relation between ${\mathbf H}_r(0)$ and  $\mathbf{S}_0(n, r)$. 
To begin with, let 
\begin{align*}
\mathbf{H}_r(q) &:= H_r(q) \otimes_{\mathbb{Z}[q]} \mathbb{C}[q],\\
\mathbf{S}_q(n, r) &:= S_q(n, r) \otimes_{\mathbb{Z}[q]} \mathbb{C}[q].
\end{align*}
Let $V_q := {\mathbb{C}[q]}^n$, and let $(\xi_i)_{1 \leq i \leq n}$ denote the standard basis for $V_q$.
There is a right action of $\mathbf{H}_r(q)$ on the tensor space ${V_q}^{\otimes r}$ defined by
\begin{align}\label{eq: Hecke right action on V^r}
\mathbf{v} \cdot T_i =
\begin{cases}
\mathbf{v}^{\sigma_i} & \text{if } k_i < k_{i+1}, \\
q\mathbf{v} & \text{if } k_i = k_{i+1}, \\
q \mathbf{v}^{\sigma_i} + (q - 1) \mathbf{v} & \text{if } k_i > k_{i+1},
\end{cases}
\end{align}
where $\mathbf{v} = \xi_{k_1} \otimes \cdots \otimes \xi_{k_r} \in V^{\otimes r}$, and $\mathbf{v}^{\sigma_i}$ denotes the tensor obtained from $\mathbf{v}$ by interchanging the $i$-th and $(i+1)$-st tensor factors.
Under this action, there is a right $\mathbf{H}_r(q)$-module isomorphism 
\begin{equation}\label{isomorphism phi q}
\phi_q: \bigoplus_{\lambda \in \Lambda(n, r)} x_\lambda {\mathbf H}_r(q) \stackrel{\cong}{\longrightarrow} {V_q}^{\otimes r}
\end{equation}
(for instance, see \cite[Corollary 3.1.5]{DD91}).
Since  
\begin{align*}
\operatorname{End}_{H_r(q)}\left (\bigoplus_{\lambda \in \Lambda(n, r)} x_\lambda  H_r(q)\right) \otimes_{\mathbb Z[q]} \mathbb C[q]  
\cong \operatorname{End}_{{\mathbf H}_r(q)}\left (\bigoplus_{\lambda \in \Lambda(n, r)} x_\lambda {\mathbf H}_r(q)\right),
\end{align*}
the isomorphism in \cref{isomorphism phi q} induces a $\mathbb C[q]$-algebra isomorphism 
\begin{align*}
\mathbf{S}_q(n, r)\cong \operatorname{End}_{{\mathbf H}_r(q)}\left ({V_q}^{\otimes r}\right).
\end{align*}

Let $V_0 := {\mathbb{C}}^n$, and let $({\bar \xi}_i)_{1 \leq i \leq n}$ denote the standard basis for $V_0$.
There is a right action of $\mathbf{H}_r(0)$ on the tensor space ${V_0}^{\otimes r}$ defined by
\begin{align}\label{eq: Hecke right action on V^r q=0}
\mathbf{v} \cdot {\bar \pi}_i =
\begin{cases}
\mathbf{v}^{\sigma_i} & \text{if } k_i < k_{i+1}, \\
0 & \text{if } k_i = k_{i+1}, \\
-\mathbf{v} & \text{if } k_i > k_{i+1},
\end{cases}
\end{align}
where $\mathbf{v} = {\bar \xi}_{k_1} \otimes \cdots \otimes {\bar \xi}_{k_r} \in V^{\otimes r}$, and $\mathbf{v}^{\sigma_i}$ denotes the tensor obtained from $\mathbf{v}$ by interchanging the $i$-th and $(i+1)$-st tensor factors.
For $\lambda \in \Lambda(n, r)$, let 
$${\bar x}_{\lambda} := \sum_{w \in \SG_{\lambda}} {\bar \pi}_w \in {\mathbf H}_r(0).$$
It follows from \cite[Lemma 4.3]{79Norton}
that ${\bar x}_{\lambda}=\pi_{w_0(\set(\lambda))}$.
Tensoring with $\mathbb C[q]/(q)$ over $\mathbb C[q]$ is right exact, the isomorphism $\phi_q$ induces a right $\mathbf{H}_r(0)$-module isomorphism 
\begin{equation}\label{isomorphism phi 0}
\phi_0:=\phi_q \otimes \operatorname{id}: \bigoplus_{\lambda \in \Lambda(n, r)} {\bar x}_\lambda {\mathbf H}_r(0) \stackrel{\cong}{\longrightarrow} {V_0}^{\otimes r}.
\end{equation}
Combining \cref{isomorphism phi 0} with the isomorphism 
\begin{align*}
{\mathbf S}_0(n, r) 
&\cong \operatorname{End}_{{\mathbf H}_r(0)}\left (\bigoplus_{\lambda \in \Lambda(n, r)} {\bar x}_\lambda {\mathbf H}_r(0)\right) \quad \text{ (by \cite[Remark 2.5]{DY12})}
\end{align*}  
yields a $\mathbb C$-algebra isomorphism 
\begin{equation}\label{Realizing Schur algebra as endomorphism algebra}    
\mathbf{S}_0(n, r)\cong \operatorname{End}_{{\mathbf H}_r(0)}\left ({V_0}^{\otimes r}\right).
\end{equation}

\subsection{The degenerate quantum group $U_0(\mathfrak{gl}_n)$ and its polynomial modules}

In~\cite{KT5}, Krob and Thibon introduced a degenerate analogue of the quantum group $U_q(\mathfrak{gl}_n)$, denoted $U_0(\mathfrak{gl}_n)$, which was further investigated by Hivert~\cite{00Hivert} and shown to arise as a specialization of the two-parameter quantum enveloping algebra defined by Takeuchi~\cite{90Takeuchi}.

\begin{definition}{\rm (\cite{KT5})}
The \emph{degenerate quantum group} $U_0(\mathfrak{gl}_n)$ is the unital $\mathbb{C}$-algebra generated by elements $\{ \mathbf{e}_i, \mathbf{f}_i \mid 1 \leq i \leq n-1 \}$ and $\{ \mathbf{k}_i \mid 1 \leq i \leq n \}$, subject to the following relations:
\begin{align*}
\mathbf{k}_i \mathbf{k}_j &= \mathbf{k}_j \mathbf{k}_i & &\text{for } 1 \leq i, j \leq n, \\
\mathbf{e}_{i-1} \mathbf{k}_i &= 0 & &\text{for } 2 \leq i \leq n-1, \\
\mathbf{k}_i \mathbf{e}_i &= 0 & &\text{for } 1 \leq i \leq n-1, \\
\mathbf{k}_i \mathbf{e}_j &= \mathbf{e}_j \mathbf{k}_i & &\text{for } j \neq i-1, i, \\
\mathbf{k}_i \mathbf{f}_{i-1} &= 0 & &\text{for } 2 \leq i \leq n-1, \\
\mathbf{f}_i \mathbf{k}_i &= 0 & &\text{for } 1 \leq i \leq n-1, \\
\mathbf{k}_i \mathbf{f}_j &= \mathbf{f}_j \mathbf{k}_i & &\text{for } j \neq i-1, i, \\
[\mathbf{e}_i, \mathbf{f}_j] &= \delta_{ij}(\mathbf{k}_{i+1} - \mathbf{k}_i) & &\text{for } 1 \leq i, j \leq n-1, \\
\mathbf{e}_i^2 \mathbf{e}_{i+1} - \mathbf{e}_i \mathbf{e}_{i+1} \mathbf{e}_i &= 0 & &\text{for } 1 \leq i \leq n-2, \\
\mathbf{e}_{i+1} \mathbf{e}_i^2 - \mathbf{e}_i \mathbf{e}_{i+1} \mathbf{e}_i &= 0 & &\text{for } 1 \leq i \leq n-2, \\
\mathbf{f}_{i+1} \mathbf{f}_i^2 - \mathbf{f}_i \mathbf{f}_{i+1} \mathbf{f}_i &= 0 & &\text{for } 1 \leq i \leq n-2, \\
\mathbf{f}_i^2 \mathbf{f}_{i+1} - \mathbf{f}_i \mathbf{f}_{i+1} \mathbf{f}_i &= 0 & &\text{for } 1 \leq i \leq n-2, \\
[\mathbf{e}_i, \mathbf{e}_j] &= 0 & &\text{for } |i - j| > 1, \\
[\mathbf{f}_i, \mathbf{f}_j] &= 0 & &\text{for } |i - j| > 1, \\
\mathbf{e}_i \mathbf{e}_{i+1} \mathbf{k}_i &= \mathbf{k}_{i+1} \mathbf{e}_i \mathbf{e}_{i+1} & &\text{for } 1 \leq i \leq n-2, \\
\mathbf{f}_{i+1} \mathbf{f}_i \mathbf{k}_{i+1} &= \mathbf{k}_i \mathbf{f}_{i+1} \mathbf{f}_i & &\text{for } 1 \leq i \leq n-2.
\end{align*}

The degenerate quantum group $U_0(\mathfrak{gl}_n)$ admits a bialgebra structure given by the comultiplication $\Delta$ and counit $\epsilon$, defined on the generators by
\begin{align*}
\Delta(\mathbf{e}_i) &= 1 \otimes \mathbf{e}_i + \mathbf{e}_i \otimes \mathbf{k}_i, & \epsilon(\mathbf{e}_i) &= 0, \\
\Delta(\mathbf{f}_i) &= \mathbf{k}_{i+1} \otimes \mathbf{f}_i + \mathbf{f}_i \otimes 1, & \epsilon(\mathbf{f}_i) &= 0, \\
\Delta(\mathbf{k}_i) &= \mathbf{k}_i \otimes \mathbf{k}_i, & \epsilon(\mathbf{k}_i) &= 1,
\end{align*}
for all $1 \leq i \leq n-1$ (and $1 \leq i \leq n$ in the case of $\mathbf{k}_i$).
\end{definition}

As before, let $({\bar \xi}_i)_{1 \leq i \leq n}$ denote the standard basis for the complex vector space $V_0 = \mathbb{C}^n$. 
For each $1 \leq i, j \leq n$, define the linear operator $E_{ij} \in \operatorname{End}_{\mathbb{C}}(V_0)$ by
\[
E_{ij}({\bar \xi}_k) = \delta_{jk} {\bar \xi}_i.
\]
These operators yield a representation of the quantum enveloping algebra $U_0(\mathfrak{gl}_n)$ via an algebra homomorphism
\[
\rho_{V_0} : U_0(\mathfrak{gl}_n) \longrightarrow \operatorname{End}_{\mathbb{C}}(V_0)
\]
given explicitly by
\[
\begin{aligned}
\rho_V(\mathbf{e}_i) &= E_{i,i+1} & &\text{for } 1 \leq i \leq n-1, \\
\rho_V(\mathbf{f}_i) &= E_{i+1,i} & &\text{for } 1 \leq i \leq n-1, \\
\rho_V(\mathbf{k}_i) &= \sum_{j \neq i} E_{j,j} & &\text{for } 1 \leq i \leq n.
\end{aligned}
\]

We briefly review the results on polynomial $U_0(\mathfrak{gl}_n)$-modules presented in \cite{KT4}.
The representation $\rho_V$ of $U_0(\mathfrak{gl}_n)$ on $V_0$ naturally extends to a $U_0(\mathfrak{gl}_n)$-module structure on the tensor space ${V_0}^{\otimes r}$ via the comultiplication.
The actions of $U_0(\mathfrak{gl}_n)$ and the $0$-Hecke algebra ${\mathbf H}_r(0)$ on ${V_0}^{\otimes r}$, as given in \cref{eq: Hecke right action on V^r}, commute with each other.
Consequently, ${V_0}^{\otimes r}$ admits the structure of a $(U_0(\mathfrak{gl}_n), {\mathbf H}_r(0))$-bimodule.
Therefore we derive a functor 
\begin{equation}\label{functor from H to U}
F_{n,r} : \text{${\mathbf H}_r(0)$-\textsf{mod}} \to \text{$U_0(\mathfrak{gl}_n)$-\textsf{mod}}, \quad M \mapsto {V_0}^{\otimes r} \otimes_{{\mathbf H}_r(0)} M.
\end{equation}

\begin{remark}
In contrast to the classical Schur--Weyl duality, the double commutant property fails for the $(U_0(\mathfrak{gl}_n), {\mathbf H}_r(0))$-bimodule ${V_0}^{\otimes r}$ (see \cite[Note~5.7]{KT5}).
\end{remark}

We call a $U_0(\mathfrak{gl}_n)$-submodule of ${V_0}^{\otimes r}$ a {\it polynomial $U_0(\mathfrak{gl}_n)$-module of degree $r$}.
For each $\alpha \in \Lambda^+(r)$, let  
$$
D_{\alpha} := V_0^{\otimes r} \otimes_{{\mathbf H}_r(0)} F_{\alpha}, \qquad 
N_{\alpha} := V_0^{\otimes r} \otimes_{{\mathbf H}_r(0)} R_{\alpha}.
$$

The module $D_{\alpha}$ admits a basis indexed by certain combinatorial objects called quasi-ribbon tableaux.
Let $n$ be a positive integer and $\alpha \in \Lambda^+(r)$. A \emph{quasi-ribbon tableau} of shape $\alpha$ is a filling of the ribbon diagram of shape $\alpha$ with entries in the alphabet $\{1, 2, \ldots, n\}$ such that the entries weakly increase from left to right in each row and strictly increase from top to bottom in each column. Let ${\rm QRT}(n, \alpha)$ denote the set of quasi-ribbon tableaux of shape $\alpha$ with entries in $\{1, 2, \ldots, n\}$.

\begin{remark} \label{When D alpha=0 ?}
By definition, $\operatorname{QRT}(n, \alpha) = \varnothing$ if 
$\ell(\alpha) > n$, and is a singleton if  $\ell(\alpha) = n$.
\end{remark}

For $T \in {\rm QRT}(n, \alpha)$, define $\mathsf{read}(T)$ to be the word obtained by reading the entries of $T$ column by column, from bottom to top and from left to right. A word $w$ over $\{1, 2, \ldots, n\}$ is called a \emph{quasi-ribbon word} of shape $\alpha$ if $w = \mathsf{read}(T)$ for some $T \in {\rm QRT}(n, \alpha)$. Let ${\rm QRW}(n, \alpha)$ denote the set of quasi-ribbon words of shape $\alpha$. 
The sets ${\rm QRT}(n, \alpha)$ and ${\rm QRW}(n, \alpha)$ are naturally in bijection via the map
\[
\mathsf{read} : {\rm QRT}(n, \alpha) \to {\rm QRW}(n, \alpha),
\]
whose inverse is denoted by
\[
\mathsf{tab} : {\rm QRW}(n, \alpha) \to {\rm QRT}(n, \alpha).
\]
Here, $\mathsf{tab}(w)$ is defined by placing the letters of $w$ into the shape $\alpha$ column by column, from bottom to top and from left to right.
For example, the word $u = 112213312$ is not a quasi-ribbon word, as its decreasing factorization does not correspond to any quasi-ribbon tableau. In contrast, the word $v = 153216767$ is a quasi-ribbon word of shape $(2,1,1,3,2)$, corresponding to the tableau
\[
\begin{ytableau}
1 & 1  \\
\none & 2 \\
\none & 3 \\
\none & 5 & 6 & 6 \\
\none  & \none & \none & 7 & 7
\end{ytableau}
\]

Following~\cite{KT5}, we define a $U_0(\mathfrak{gl}_n)$-module structure on the $\mathbb{C}$-span of ${\rm QRT}(n, \alpha)$ using Kashiwara crystal operators $\tilde{e}_i$ and $\tilde{f}_i$, defined on words in $\{1, 2, \ldots, n\}$ as follows.
Given $w \in {\rm QRW}(n, \alpha)$, let $w^{(i)}$ denote the subword consisting of the letters $i$ and $i+1$, in their original order. Perform the cancellation process that removes every adjacent pair $(i+1)i$ from $w^{(i)}$, proceeding from left to right. The result is a reduced word of the form $i^r(i+1)^s$ for some $r,s \ge 0$. 
The operator $\tilde{e}_i$ changes the leftmost occurrence of $i+1$ in the reduced subword to $i$, while $\tilde{f}_i$ changes the rightmost $i$ to $i+1$. If no such letter exists (i.e., if $s = 0$ for $\tilde{e}_i$ or $r = 0$ for $\tilde{f}_i$), the result is defined to be zero.
The induced $U_0(\mathfrak{gl}_n)$-action on ${\rm QRT}(n,\alpha)$ is given by:
\[
\mathbf{f}_i \cdot T = 
\begin{cases} 
\mathsf{tab}(\tilde{f}_i(\mathsf{read}(T))) & \text{if } \tilde{f}_i(\mathsf{read}(T)) \in {\rm QRW}(n, \alpha), \\
0 & \text{otherwise},
\end{cases}
\]
and similarly,
\[
\mathbf{e}_i \cdot T = 
\begin{cases} 
\mathsf{tab}(\tilde{e}_i(\mathsf{read}(T))) & \text{if } \tilde{e}_i(\mathsf{read}(T)) \in {\rm QRW}(n, \alpha), \\
0 & \text{otherwise}.
\end{cases}
\]
The action of $\mathbf{k}_i$ is defined by
\begin{align}\label{eq: k action on QR}
\mathbf{k}_i \cdot T = 
\begin{cases}
T & \text{if $T$ has no entry equal to $i$}, \\
0 & \text{otherwise}.
\end{cases}
\end{align}
This equips $\mathbb{C}[\mathrm{QRT}(n, \alpha)]$, the $\mathbb{C}$-span of $\mathrm{QRT}(n, \alpha)$, with the structure of a $U_0(\mathfrak{gl}_n)$-module.

\begin{proposition}{\rm (\cite[Proposition 7.1]{KT5})}  \label{realization of D(alpha)} 
Let $\alpha \in \Lambda^+(r)$.
With the above $U_0(\mathfrak{gl}_n)$-action, 
$\mathbb{C} [{\rm QRT}(n, \alpha)]$ is isomorphic to $D_{\alpha}$ as $U_0(\mathfrak{gl}_n)$-modules.
\end{proposition}

From \cref{When D alpha=0 ?}, it follows that \(D_\alpha = 0\) if \(\ell(\alpha) > n\), and \(D_\alpha\) is a one-dimensional trivial \(U_0(\mathfrak{gl}_n)\)-module if \(\ell(\alpha) = n\).

\begin{proposition}{\rm (cf. \cite[Proposition 7.3]{KT5})}
\label{Krob thibon classification}
Let $n,r$ be nonnegative integers.
Then we have

\begin{enumerate}[label = {\rm (\arabic*)}]
    \item 
Every simple polynomial $U_0(\mathfrak{gl}_n)$-module of degree $r$ is isomorphic to  $D_{\alpha}$ for some $\alpha \in  \Lambda^+(r)$ with $\ell(\alpha)\le n$.

    \item Every indecomposable polynomial $U_0(\mathfrak{gl}_n)$-module of degree $r$ is isomorphic to  $N_{\alpha}$ for some $\alpha \in  \Lambda^+(r)$.
\end{enumerate}
\end{proposition}

By \cite[Proposition 7.1]{KT5}, the hypoplactic character of $D_{\alpha}$ is equal to the sum of all quasi-ribbon words of shape $\alpha$ over $x_{11}< \cdots < x_{nn}$, and its image in the ring of quasi-symmetric functions is the fundamental quasi-symmetric function $F_\alpha$. And, by \cite[Proposition 7.1]{KT5}, the hypoplactic character of $N_{\alpha}$ is equal to the sum of all words of ribbon shape $\alpha$ over $x_{11}< \cdots < x_{nn}$, and its image in the the ring of noncommutative symmetric functions is the ribbon Schur function $R_\alpha$. 
It follows that for 
$\alpha, \beta \in \Lambda^+(r)$ with  $\ell(\alpha), \ell(\beta) \le n$ and 
$\alpha \ne \beta$, we have $D_\alpha \not\cong D_\beta$ and $N_\alpha \not\cong N_\beta$.

Unlike $D_\alpha$, 
$N_\alpha$ may not be zero even when $\ell(\alpha) > n$.
Since $F_{n,r}$ is exact, 
every simple quotient arising from a composition series of $N_\alpha$
is of the form $D_\beta$ for some $\beta\in \Lambda^+(r)$.
It implies that  
$$[N_\alpha:D_\beta]=
\begin{cases}
 [R_\alpha:F_\beta] & \text{ if }\ell(\beta)\le n,\\
0 & \text{ otherwise,}
\end{cases}
$$
and therefore 
$$\ch(N_\alpha)=\sum_{\beta\in \bigcup_{0\le l\le n}\Lambda^+(l,r)}[R_\alpha:F_\beta]\ch(D_\beta).$$

\subsection{
Functorial relations among 
$U_0(\mathfrak{gl}_n)$-\textsf{mod},  
$\mathbf{S}_0(n,r)$-\textsf{mod}, and $\mathbf{H}_r(0)$-\textsf{mod}}

In the previous subsection, we introduced a functor
\begin{equation*}
F_{n,r} : \text{${\mathbf H}_r(0)$-\textsf{mod}} \to \text{$U_0(\mathfrak{gl}_n)$-\textsf{mod}}.
\end{equation*}
Here, we introduce two more functors 
\begin{align*}
&G_{n,r} : {\mathbf H}_r(0)\text{-\textsf{mod}} \longrightarrow {\mathbf S}_0(n,r)\text{-\textsf{mod}},\\
&\Psi_{n,r}: {\mathbf S}_0(n,r)\text{-\textsf{mod}} \to U_0(\mathfrak{gl}_n)\text{-\textsf{mod}}.
\end{align*}
First, observe that the actions of ${\mathbf H}_r(0)$ and $\operatorname{End}_{{\mathbf H}_r(0)}({V_0}^{\otimes r})$ on ${V_0}^{\otimes r}$ commute. Hence, the tensor space ${V_0}^{\otimes r}$ admits the structure of a $(\operatorname{End}_{{\mathbf H}_r(0)}({V_0}^{\otimes r}), {\mathbf H}_r(0))$-bimodule. 
Identifying ${\mathbf S}_0(n,r)$ with $\operatorname{End}_{{\mathbf H}_r(0)}({V_0}^{\otimes r})$ as in \cref{Realizing Schur algebra as endomorphism algebra}, this yields a functor
\[
G_{n,r} : {\mathbf H}_r(0)\text{-\textsf{mod}} \longrightarrow {\mathbf S}_0(n,r)\text{-\textsf{mod}}, \quad M \mapsto {V_0}^{\otimes r} \otimes_{{\mathbf H}_r(0)} M.
\]
Next, observe that the actions of $U_0(\mathfrak{gl}_n)$ and ${\mathbf H}_r(0)$ on ${V_0}^{\otimes r}$ commute. 
Hence, the image of the representation of $U_0(\mathfrak{gl}_n)$ in $\operatorname{End}_{\mathbb C}({V_0}^{\otimes r})$ lies in the subalgebra $\operatorname{End}_{{\mathbf H}_0(r)}({V_0}^{\otimes r})$ by definition.
We denote the induced algebra homomorphism by
\begin{equation}\label{def of psi}
\psi_{n,r} : U_0(\mathfrak{gl}_n) \to {\mathbf S}_0(n, r).
\end{equation}
It should be noted that, in contrast to the generic case, $\psi_{n,r}$ is not surjective in general (see \cite[Note~5.7]{KT5}).
Then, we define 
$$\Psi_{n,r}: {\mathbf S}_0(n,r)\text{-\textsf{mod}} \to U_0(\mathfrak{gl}_n)\text{-\textsf{mod}}$$ 
to be the pullback functor induced by the map $\psi_{n,r}: U_0(\mathfrak{gl}_n) \to {\mathbf S}_0(n,r)$.
It is straightforward to verify that
\begin{align}\label{eq: phi G = F}
\Psi_{n,r} \circ G_{n,r} = F_{n,r}.
\end{align}
More precisely, for any finite-dimensional ${\mathbf H}_r(0)$-module $M$, the following diagram commutes:
\[
\begin{tikzcd}[row sep=large]
 & M \arrow[dl, "G_{n,r}"', long mapsto] \arrow[dr, "F_{n,r}", long mapsto] & \\
{V_0}^{\otimes r} \otimes_{{\mathbf H}_r(0)} M \arrow[rr, "\Psi_{n,r}"] && {V_0}^{\otimes r} \otimes_{{\mathbf H}_r(0)} M
\end{tikzcd}
\]
Let $\alpha \in \Lambda^+(r)$ with $\ell(\alpha) \le n$.
Combining the commutativity of the diagram with \cref{Krob thibon classification}, it follows that $G_{n,r}(F_{\alpha})$ is a simple $\mathbf{S}_0(n, r)$-module, and $G_{n,r}(R_{\alpha})$ is an indecomposable $\mathbf{S}_0(n, r)$-module.

\subsubsection{Projectivity of $G_{n,r}(R_{\alpha})$}
We here show that, for each $\alpha \in \Lambda^+(r)$ with $\ell(\alpha)\le n$, the module $G_{n,r}(R_{\alpha})$ is projective over $\mathbf{S}_0(n, r)$.
Let $\{e_\alpha \mid \alpha \in \Lambda^+(r)\}$ be a complete set of primitive orthogonal idempotents of ${\mathbf H}_r(0)$ such that $R_\alpha = {\mathbf H}_r(0) e_\alpha$.
Let 
\begin{equation}\label{condition of e}
e := \sum_{\substack{\alpha \in \Lambda^+(r)\\ \ell(\alpha) \le n}} e_\alpha,
\end{equation}
It was remarked in \cite[Section 4]{DY12} that $\mathbf{S}_0(n, r)$ is Morita equivalent to $e{\mathbf H}_r(0)e$.
Let us explain this in more detail.
Put 
$$Q:=\operatorname{End}_{{\mathbf H}_r(0)}(e{\mathbf H}_r(0),{V_0}^{\otimes r}).$$
$\operatorname{End}_{{\mathbf H}_r(0)}(e{\mathbf H}_r(0))$ acts on $Q$ from the right via post-composition, while 
$\operatorname{End}_{{\mathbf H}_r(0)}({V_0}^{\otimes r})$ acts on $Q$ from the left via pre-composition, and thus 
$Q$ is an $(\mathbf{S}_0(n, r),\operatorname{End}_{{\mathbf H}_r(0)}(e{\mathbf H}_r(0)))$-bimodule.
Define
$$
\Phi:\End_{{\mathbf H}_r(0)}({V_0}^{\otimes r})\longrightarrow \End_{\operatorname{End}_{{\mathbf H}_r(0)}(e{\mathbf H}_r(0))}\big(Q\big)
$$
by
$$
\Phi(\varphi)(g):=\varphi\circ g\qquad(\varphi\in\End_{{\mathbf H}_r(0)}({V_0}^{\otimes r}),\ g\in Q).
$$
One can easily see that it is a $\mathbb C$-algebra isomorphism.
Recall that $$\operatorname{End}_{{\mathbf H}_r(0)}(e{\mathbf H}_r(0))\cong e{\mathbf H}_r(0)e\quad (\text{as $\mathbb C$-algebras}).$$  
Via the $\mathbb C$-vector space isomorphism 
$$Q \stackrel{\cong }{\to} {V_0}^{\otimes r}e, \quad \varphi \mapsto \varphi(e), $$
let us view ${V_0}^{\otimes r}e$ 
as an $(\mathbf{S}_0(n, r),e\mathbf {H}_r(0)e)$-bimodule.
Define
$$
\Phi:\End_{{\mathbf H}_r(0)}({V_0}^{\otimes r})\longrightarrow \End_{e{\mathbf H}_r(0)e}\big({V_0}^{\otimes r}e\big)
$$
by
$$
\Phi(\varphi)(xe):=\varphi(x)e \qquad(\varphi\in\End_{{\mathbf H}_r(0)}({V_0}^{\otimes r}),\ x\in {V_0}^{\otimes r}).
$$
One can easily see that it is a $\mathbb C$-algebra isomorphism.  
Endow ${V_0}^{\otimes r}e$ with a left $e{\mathbf H}_r(0)e$-action 
$$h\cdot x:= x \cdot \chi(h) \quad (h\in e{\mathbf H}_r(0)e, x\in {V_0}^{\otimes r}e).$$
The resulting left $e{\mathbf H}_r(0)e$-module is denoted by $({V_0}^{\otimes r}e)^{\mathrm{op}}$.
Then  
$$
{\mathbf S}_0(n,r) \cong \End_{e{\mathbf H}_r(0)e}\big(({V_0}^{\otimes r}e)^{\mathrm{op}}\big)^{\mathrm{op}}.
$$
To see that $({V_0}^{\otimes r}e)^{\mathrm{op}}$ is a progenerator of $(e{\mathbf H}_r(0)e)$-mod, let 
$$
{V_0}^{\otimes r}\cong \bigoplus_{\lambda \in \Lambda(n,r)} x_{\lambda}{\mathbf H}_r(0) 
   \cong \bigoplus_{\alpha \in \Lambda^+(r)} \bigl(e_\alpha {\mathbf H}_r(0)\bigr)^{d_\alpha} \quad (\text{as right ${\mathbf H}_r(0)$-modules}),
$$
where $d_\alpha$ is the multiplicity of $e_\alpha {\mathbf H}_r(0)$.
As right $(e{\mathbf H}_r(0)e)$-modules, 
we have the isomorphism 
$$
{V_0}^{\otimes r}e
   \cong \bigoplus_{\alpha \in \Lambda^+(r)} \bigl(e_\alpha {\mathbf H}_r(0)e\bigr)^{d_\alpha}.
$$
From the $e{\mathbf H}_r(0)e$-module isomorphism  
$$ (e_\alpha{\mathbf H}_r(0)e)^{\mathrm{op}} \stackrel{\cong}{\to}  e{\mathbf H}_r(0)e_\alpha, \quad e_{\alpha}ae \mapsto e\chi(a)e_{\alpha},$$
it follows that 
$$
({V_0}^{\otimes r}e)^{\mathrm{op}}
   \stackrel{\cong}{\to} \bigoplus_{\alpha \in \Lambda^+(r)} \bigl(e{\mathbf H}_r(0)e_\alpha \bigr)^{d_\alpha}.
$$
It was shown in \cite[Proposition 4.2]{DY12} that 
$$
d_{\alpha}= \bigl|\{\lambda \in \Lambda(n,r) \mid \set(\alpha)\subseteq \set(\lambda^+)\}\bigr|
$$
and $d_\alpha>0$ precisely when $\ell(\alpha)\le n$.
It follows that 
$({V_0}^{\otimes r}e)^{\mathrm{op}}$ is a progenerator for $e{\mathbf H}_r(0)e$-mod. 

For the subsequent discussion, however, we shall work with the progenerator $({V_0}^{\otimes r}e)^*$ for $e{\mathbf H}_r(0)e$-mod, rather than $({V_0}^{\otimes r}e)^{\mathrm{op}}$. 
To this end, following \cite[Section~4]{05Fayers}, consider the standard Frobenius linear functional
$$
\varepsilon: {\mathbf H}_r(0)\to \mathbb C,\qquad 
\varepsilon(\pi_w)=
\begin{cases}
1,& w=w_0,\\
0,& w\ne w_0,
\end{cases}
\quad (w\in \SG_r).
$$
Define a $\mathbb C$-linear map
$$
\Theta:\ e{\mathbf H}_r(0)\,\autophi(e_\alpha)\ \longrightarrow\ (\,e_\alpha  {\mathbf H}_r(0) e\,)^*,
$$
by 
$\Theta(eh\,\autophi(e_\alpha))(x)\ 
=\ \varepsilon(xh)$ for all $x\in e_\alpha {\mathbf H}_r(0) e$ and $h\in {\mathbf H}_r(0)$.

\begin{lemma}
$\Theta$ is a well defined left $e{\mathbf H}_r(0)e$-module isomorphism.
\end{lemma}

\begin{proof}    
First, we prove that $\Theta$ is well defined.
Suppose that $eh\,\autophi(e_\alpha)=eh'\,\autophi(e_\alpha)$. For $x\in e_\alpha {\mathbf H}_r(0) e$, since $x=e_\alpha x e$, one has
$$
\varepsilon(xh)=\varepsilon((e_\alpha x e)h)=\varepsilon(e_\alpha (xeh)).
$$
By \cite[Proposition~4.2]{05Fayers}, the map
$$
\autophi: {\mathbf H}_r(0)\longrightarrow {\mathbf H}_r(0), 
\qquad \pi_i \mapsto \pi_{r-i}\quad (1\le i\le r-1),
$$
is the Nakayama automorphism of ${\mathbf H}_r(0)$ associated with $\varepsilon$. Hence,
$$
\varepsilon(xh)=\varepsilon(xeh\,\autophi(e_\alpha))
=\varepsilon(xeh'\,\autophi(e_\alpha))
=\varepsilon(e_\alpha xeh')=\varepsilon(xh').
$$
Thus $\Theta$ is well defined.

Next, we prove that $\Theta$ is a left $e{\mathbf H}_r(0)e$-module homomorphism.
For $y\in e{\mathbf H}_r(0)e$, $h\in {\mathbf H}_r(0)$ and $x\in e_\alpha {\mathbf H}_r(0)e$, we have 
\begin{align*}
\Theta\bigl(y\cdot (eh\autophi^{-1}(e_\alpha))\bigr)(x)
&= \Theta\bigl(yh\autophi^{-1}(e_\alpha)\bigr)(x)
= \varepsilon\bigl(x\,(yh)\bigr)
= \varepsilon\bigl((xy)\,h\bigr)\\
&= \Theta\bigl(eh\autophi^{-1}(e_\alpha)\bigr)(xy)
= \bigl(y\cdot\Theta(eh\autophi^{-1}(e_\alpha))\bigr)(x).
\end{align*}

Finally, we show that $\Theta$ is a $\mathbb C$-isomorphism.
Suppose $\Theta(eh\autophi^{-1}(e_\alpha))=0$. Then
$\varepsilon(xh)=0$ for all 
$x\in e_\alpha e$.
By the Nakayama property $\varepsilon(ab)=\varepsilon(b\autophi(a))$, this is equivalent to
$$
\varepsilon(h\,\autophi(x))=0\qquad \text{for all }x\in e_\alpha {\mathbf H}_r(0) e.
$$
Hence
$
\varepsilon(hy)=0$ for all $y\in \autophi(e_\alpha){\mathbf H}_r(0)e$.
Since the Frobenius form $(u,v)\mapsto \varepsilon(uv)$ is nondegenerate, the condition above forces
$h\autophi(e_\alpha)e=0$.
Thus $\ker\Theta=0$, so $\Psi$ is injective.
Now, the assertion follows since both spaces have equal dimension. 
\end{proof}

From the isomorphism    
${\mathbf H}_r(0)\autophi(e_\alpha)\cong {\mathbf H}_r(0)e_{\alpha^{\rmr}}$,
we obtain the left $e{\mathbf H}_r(0)e$-module isomorphisms
$$({V_0}^{\otimes r}e)^*\cong \bigoplus_{\alpha \in \Lambda^+(r)}((e_\alpha {\mathbf H}_r(0)e)^{d_\alpha})^* \cong \bigoplus_{\alpha \in \Lambda^+(r)}(e{\mathbf H}_r(0)e_{\alpha^\rmr})^{d_{\alpha}}.$$
It shows that $({V_0}^{\otimes r}e)^*$ is a  progenerator for $e{\mathbf H}_r(0)e$-mod
and 
\begin{equation*}  
\mathbf{S}_0(n, r)\cong \operatorname{End}_{e{\mathbf H}_r(0)e}\left (({V_0}^{\otimes r}e)^*\right)^{\rm op}.
\end{equation*}
This gives rise to the following equivalence of $e{\mathbf H}_r(0)e\text{-mod}$ and ${\mathbf S}_0(n,r)\text{-mod}$: 
$$\operatorname{Hom}_{e{\mathbf H}_r(0)e}(({V_0}^{\otimes r}e)^*,e{\mathbf H}_r(0)e) \otimes_{e{\mathbf H}_0(r)e} -:e{\mathbf H}_r(0)e\text{-\textsf{mod}} \to {\mathbf S}_0(n,r)\text{-\textsf{mod}}.$$
This functor turns out be naturally isomorphic to 
$$A_{n,r}:={V_0}^{\otimes r}e\otimes_{e{\mathbf H}_r(0)e}-: e{\mathbf H}_r(0)e\text{-\textsf{mod}} \to {\mathbf S}_0(n,r)\text{-\textsf{mod}}.$$
Indeed, given an object $M$ in $e{\mathbf H}_r(0)e\text{-mod}$,
we have $\mathbf{S}_0(n, r)$-module isomorphisms
\begin{align*}
&\operatorname{Hom}_{e{\mathbf H}_r(0)e}(({V_0}^{\otimes r}e)^*,e{\mathbf H}_r(0)e) \otimes_{e{\mathbf H}_r(0)e} M \\
&\cong {V_0}^{\otimes r}e \otimes_{e{\mathbf H}_r(0)e}e{\mathbf H}_r(0)e\otimes_{e{\mathbf H}_r(0)e} M \\
&\cong {V_0}^{\otimes r}e\otimes_{e{\mathbf H}_r(0)e}M.
\end{align*}
Consider the functor 
$$B_{n,r}: {\mathbf H}_r(0)\text{-\textsf{mod}} \to e{\mathbf H}_r(0)e\text{-\textsf{mod}}, \quad M \mapsto eM, $$
which is naturally isomorphic to $\operatorname{Hom_{{\mathbf H}_r(0)}}({\mathbf H}_r(0)e, -)$.

\begin{theorem} \label{how to prove projectivity}
The following hold:
\begin{enumerate}[label = {\rm (\arabic*)}]
   \item 
   $B_{n,r}\circ A_{n,r}$ and $G_{n,r}$ are naturally isomorphic.
   \item For each $\alpha \in \Lambda^+(r)$ with $\ell(\alpha)\le n$, $e{\mathbf H}_r(0)e_\alpha$ is a projective indecomposable $e{\mathbf H}_r(0)e$-module.
   \item 
For each $\alpha\in \Lambda^+(r)$ with $\ell(\alpha)\le n$,
$G_{n,r}(R_{\alpha})$ is a projective indecomposable $\mathbf{S}_0(n, r)$-module.   
\end{enumerate}
\end{theorem}

\begin{proof}
(1) We will define a natural transformation  
$\eta: B_{n,r}\circ A_{n,r} \to G_{n,r}$ as follows: 
 
Given an object $M$ in ${\mathbf H}_r(0)\text{-mod}$, define 
$$\eta_M: {V_0}^{\otimes r}e\otimes_{e{\mathbf H}_r(0)e}eM \to {V_0}^{\otimes r}\otimes_{{\mathbf H}_r(0)}M, \quad ve\otimes em \mapsto v\otimes m. 
$$
One can see that $\eta_M$is an $S_0(n,r)$-module isomorphism and its inverse is given by $v\otimes m \mapsto ve\otimes em$.
Also, given an ${\mathbf H}_r(0)$-homomorphism $f: M \to N$, the following diagram
\[ \begin{tikzcd}
{V_0}^{\otimes r}e\otimes_{e{\mathbf H}_r(0)e}eM  \arrow{r}{\eta_M} \arrow[swap]{d}{{\rm id} \otimes f} & {V_0}^{\otimes r}\otimes_{{\mathbf H}_r(0)}M \arrow{d}{{\rm id} \otimes f} \\%
{V_0}^{\otimes r}e\otimes_{e{\mathbf H}_r(0)e}eN \arrow{r}{\eta_N}& {V_0}^{\otimes r}\otimes_{{\mathbf H}_r(0)}N
\end{tikzcd}
\]
is commutative since 
\begin{align*}
&({\rm id} \otimes f) \circ \eta_M(ve\otimes em)=({\rm id} \otimes f)(v\otimes m)=v\otimes f(m), \\ 
& \eta_N\circ({\rm id} \otimes f)  (ve\otimes em)=\eta_N(ve\otimes ef(m))=v\otimes f(m).
\end{align*}

(2) From 
$$e{\mathbf H}_r(0)e=\displaystyle\bigoplus_{\alpha\in \bigcup_{0\le l\le n}\Lambda^+(l,r)}e{\mathbf H}_r(0)e_\alpha$$
it follows that $e{\mathbf H}_r(0)e_\alpha$ for $\ell(\alpha)\le n$ is a  
projective $e{\mathbf H}_r(0)e$-module.
Now, note that  
\begin{equation}  \label{why projective indecomposable}
G_{n,r}(R_{\alpha})=B_{n,r}\circ A_{n,r}({\mathbf H}_r(0)e_\alpha)=B_{n,r}(e{\mathbf H}_r(0)e_\alpha).
\end{equation}
Since $G_{n,r}(R_{\alpha})$ is indecomposable, it follows that  $e{\mathbf H}_r(0)e_\alpha$ is also indecomposable.

(3) Due to \cref{why projective indecomposable}, the assertion follows from (2)
\end{proof}

\subsubsection{
Images of simple 
and projective indecomposable $\mathbf{H}_r(0)$-modules under the functor $G_{n,r}$
}

We begin with some preparatory definitions and lemmas.

\begin{definition}
Let $\alpha \in \Lambda^+(r)$ with $\ell(\alpha) \le n$. For each $T \in \operatorname{QRT}(n, \alpha)$, define its weight vector $\operatorname{wt}(T)=(\operatorname{wt}(T)_1, \operatorname{wt}(T)_2, \ldots, \operatorname{wt}(T)_n) \in \mathbb{Z}_{\ge 0}^n$ by
$$
\operatorname{wt}(T)_i := \#\{\, \text{entries equal to } i \text{ in } T \,\} \quad \text{for } 1 \le i \le n.
$$

\end{definition}

Recall that for $\alpha \in \Lambda^+(r)$ with $\ell(\alpha) \le n$, the composition $\alpha^{\bullet} \in \Lambda^{\bullet}(n, r)$ is obtained by appending zeros to $\alpha$ so that its length is equal to $n$ (see \cref{relation between weak and strong compositions}).

\begin{lemma}\label{lem: wt bijection}
For $\alpha \in \Lambda^+(r)$ with $\ell(\alpha) \le n$, the map
\[
\wt: {\rm QRT}(n, \alpha) \to \Lambda(n,r)_{\preceq \alpha^{\bullet}}, \quad T \mapsto \wt(T)
\]
is a bijection. For the definition of the codomain, see \cref{eq: refinement}.
\end{lemma}

\begin{proof}
It follows directly from the definition of quasi-ribbon tableaux that $\wt(T) \in \Lambda(n,r)_{\preceq \alpha^{\bullet}}$ for any $T \in {\rm QRT}(n, \alpha)$.

To construct the inverse, let $\alpha = (\alpha_1, \ldots, \alpha_{\ell(\alpha)})$ and let $\mu = (\mu_1, \ldots, \mu_n) \in \Lambda(n,r)_{\preceq \alpha^{\bullet}}$. Since $\mu \preceq \alpha^{\bullet}$, there exists a unique sequence of indices
\[
1 = i_1 < i_2 < \cdots < i_{\ell(\alpha)+1} = n+1
\]
such that
\[
\alpha_k = \mu_{i_k} + \mu_{i_k+1} + \cdots + \mu_{i_{k+1}-1}
\quad \text{for all } 1 \le k \le \ell(\alpha),
\]
and $\mu_{i_k} > 0$ for all $2 \le k \le \ell(\alpha)$.
Define the tableau $T_\mu$ by filling the $k$-th row with the entries
\[
i_k^{\mu_{i_k}},\ (i_k + 1)^{\mu_{i_k + 1}},\ \ldots,\ (i_{k+1} - 1)^{\mu_{i_{k+1} - 1}},
\]
that is, the letter $j$ appears $\mu_j$ times in the $k$-th row for each $i_k \le j < i_{k+1}$.
Then the map $\mu \mapsto T_\mu$ defines the inverse of $\wt$, completing the proof.
\end{proof}

\begin{lemma}\label{lem: phi}
The map $\psi_{n,r} : U_0(\mathfrak{gl}_n) \to {\mathbf S}_0(n, r)$ in \cref{def of psi} satisfies the following identities: 
\[
\begin{aligned}
\psi_{n,r}(\mathbf{e}_i) &= e_{i} & &\text{for } 1 \leq i \leq n-1, \\
\psi_{n,r}(\mathbf{f}_i) &= f_{i} & &\text{for } 1 \leq i \leq n-1, \\
\psi_{n,r}(\mathbf{k}_i) &= \sum_{\mu_i = 0} k_{\mu} & &\text{for } 1 \leq i \leq n.
\end{aligned}
\]
\end{lemma}

The proof of \cref{lem: phi} appears in \cref{sec: proof of psi_nr}.

From the natural equivalence in \cref{eq: phi G = F}, it follows that $\Psi_{n,r}(G_{n,r}(F_{\alpha})) \cong D_{\alpha}$.
The following proposition establishes that, for each $\lambda \in \Lambda^{\bullet}(n,r)$, the modules $G_{n,r}(F_{\alpha})$ and $S_{\lambda}$ have isomorphic images under $\Psi_{n,r}$.

\begin{proposition}\label{thm: 0-Schur functor degenerate quantum group}
For each $\lambda \in \Lambda^{\bullet}(n,r)$, we have  $U_0(\mathfrak{gl}_n)$-module isomorphisms
\[
\Psi_{n,r}(S_{\lambda}) \cong D_{\lambda^{+}}.
\]
\end{proposition}

\begin{proof}
Note that $(\lambda^{+})^{\bullet} = \lambda$. Hence, by Proposition~\cref{prop: row bijection} and Lemma~\cref{lem: wt bijection}, we obtain a bijection
\[
\row^{-1} \circ \wt : {\rm QRT}(n, \lambda^{+}) \to \cb(\lambda).
\]
Let $\Gamma: D_{\lambda} \to \Psi_{n,r}(S_{\lambda})$ be the linear map induced by extending this bijection between the respective bases.

Let $T \in {\rm QRT}(n, \lambda^{+})$. 
Since $\wt(\mathbf{f}_i \cdot T) = \wt(T) -  (\underbrace{0, \ldots, 0}_{i-1}, -1, 1, 0, \ldots, 0)$, we have $\Gamma(\mathbf{f}_i \cdot T) = \mathbf{f}_i  \cdot \Gamma(T)$ by \cref{eq: row vector f action} and \cref{lem: phi}. 
Similarly, $\Gamma(\mathbf{e}_i \cdot T) = \mathbf{e}_i  \cdot \Gamma(T)$.

By \cref{eq: k action on QR}, we have 
\[
\mathbf{k}_i \cdot T = 
\begin{cases}
    T    &   \mbox{if $\wt(T)_i = 0$}, \\ 
    0    &    \mbox{otherwise.}
\end{cases}
\]
On the other hand, for $A \in \cb(\lambda)$, it follows from~\cref{eq: k_lambda act on e_A} that
\[
\sum_{\mu_i = 0} k_\mu \cdot \overline{e_A} = 
\begin{cases}
    \overline{e_A} & \text{if } \ro(A)_i = 0, \\ 
    0 & \text{otherwise}.
\end{cases}
\]
This shows $\Gamma(\mathbf{k}_i \cdot T) =   \mathbf{k}_i \cdot \Gamma(T)$ by \cref{lem: phi}.
Hence, $\Gamma$ is an isomorphism.
\end{proof}

Now, we are ready to state the main theorem.

\begin{theorem}\label{cor: 0-Schur functor 0-Hecke}
    For $\alpha \in \Lambda^{+}(r)$ with $\ell(\alpha)\le n$, we have the ${\mathbf S}_0(n,r)$-module isomorphisms
    \begin{align*}
    G_{n,r}(F_{\alpha}) \cong S_{\alpha^{\bullet}} \quad \text{ and } \quad 
    G_{n,r}(R_{\alpha}) \cong P_{\alpha^{\bullet}}.
    \end{align*}
    
\end{theorem}

\begin{proof}    
By \cref{eq: phi G = F}, we see that $G_{n,r}(F_{\alpha})$ is simple.
Assume that $G_{n,r}(F_{\alpha}) \cong S_{\beta^{\bullet}}$
for some $\beta\in \Lambda^+(r)$ with $\ell(\beta) \le n$.
By \cref{thm: 0-Schur functor degenerate quantum group}, we see that 
$$ D_\alpha \cong \Psi_{n,r}(G_{n,r}(F_{\alpha})) \cong \Psi_{n,r}(S_{\beta^{\bullet}})\cong D_\beta$$ as $U_0(\mathfrak{gl}_n)$-modules.
This shows that $\alpha=\beta$.

Let $\pi^1: R_\alpha \twoheadrightarrow F_\alpha$ be the projective cover (hence an essential epimorphism).
By right exactness of $G_{n,r}$, applying $G_{n,r}$ to the exact sequence
\[
R_\alpha \xrightarrow{\ \pi^1\ } F_\alpha \longrightarrow 0
\]
yields an exact sequence
\[
G_{n,r}(R_\alpha)\xrightarrow{\,G_{n,r}(\pi^1)\,} G_{n,r}(F_\alpha)\longrightarrow 0,
\]
so $G_{n,r}(\pi^1): G_{n,r}(R_\alpha)\twoheadrightarrow G_{n,r}(F_\alpha)\cong S_{\alpha^{\bullet}}$ is an epimorphism.
Set \(Q:=G_{n,r}(R_\alpha)\).
By \cref{how to prove projectivity}, $Q=G_{n,r}(R_{\alpha})$ is a projective indecomposable $S_0(n,r)$-module.

Let $\pi^2: P_{\alpha^{\bullet}} \twoheadrightarrow S_{\alpha^{\bullet}}$ be the projective cover (hence an essential epimorphism).
By the universal property of projective covers, there exists a morphism
\[
g: Q\longrightarrow P_{\alpha^{\bullet}}\quad\text{such that}\quad \pi^2\circ g = G_{n,r}(\pi^1).
\]
Since $\pi^2$ is essential and $\pi^2\circ g$ is surjective, it follows that $g$ is surjective.
As $P_{\alpha^{\bullet}}$ is projective, the surjection $g$ splits, so $P_{\alpha^{\bullet}}$ is a direct summand of $Q$.
Since $Q$ is indecomposable, we conclude $Q \cong P_{\alpha^{\bullet}}$.

And, by \cref{how to prove projectivity}, 
we see that $G_{n,r}(R_{\alpha})$ is a projective indecomposable $S_0(n,r)$-module.
Thus, the desired result can be obtained in the same manner as above.
\end{proof}

The following corollary follows immediately from the combination of \cref{cor: 0-Schur functor 0-Hecke} and \cref{eq: phi G = F}.

\begin{corollary}\label{thm: 0-Schur functor degenerate quantum group PIM}
For each $\lambda \in \Lambda^{\bullet}(n,r)$, we have  $U_0(\mathfrak{gl}_n)$-module isomorphisms
\[
\Psi_{n,r}(P_{\lambda}) \cong N_{\lambda^{+}}.
\]
\end{corollary}

As a second corollary of \cref{cor: 0-Schur functor 0-Hecke}, we determine the socles of  the projective indecomposable $\mathbf{S}_{n,r}$-modules.
Let $A$ be a $\mathbb{C}$-algebra and $M$ a finite-dimensional $A$-module. The \emph{socle} of $M$, denoted $\operatorname{soc}(M)$, is defined as the sum of all simple submodules of $M$.
For $\lambda = (\lambda_1,\lambda_2, \ldots, \lambda_l, 0,\ldots, 0)$ in $\Lambda^{\bullet}(n,r)$, let 
$\lambda^{\rm r} = (\lambda_l,\ldots, \lambda_{2},\lambda_1, 0,\ldots, 0)$ and thus 
$\lambda^{\rm r}=((\lambda^+)^{\rm r})^{\bullet}$.

\begin{corollary} \label{socle of Pim of Schur}
    Let $\lambda \in \Lambda^{\bullet}(n,r)$. Then we have the ${\mathbf S}_0(n,r)$-module isomorphism
    \begin{align*}
    \mathrm{soc}(P_{\lambda}) \cong S_{\lambda^r}.
    \end{align*}
\end{corollary}

\begin{proof}
Since ${V_0}^{\otimes r}$ is a projective 
${\mathbf H}_0(r)$-module, the functor
$$
G_{n,r} := {V_0}^{\otimes r} \otimes_{{\mathbf H}_0(r)} (-)
$$
is exact. 
Thus \cref{cor: 0-Schur functor 0-Hecke} implies that
$$
G_{n,r}(\mathrm{soc}(R_{\lambda^+})) \subseteq \mathrm{soc}(G_{n,r}(R_{\lambda^+})).
$$
Using the isomorphisms $\mathrm{soc}(R_{\lambda^+}) \cong F_{(\lambda^+)^{\mathrm{r}}}$ and $ {V_0}^{\otimes r} \otimes_{{\mathbf H}_0(r)} R_{\lambda^+}\cong P_{\lambda}$, it follows that
$$
S_{((\lambda^+)^{\rm r})^{\bullet}}=S_{\lambda^{\mathrm{r}}} \hookrightarrow \mathrm{soc}(P_\lambda).
$$
Since ${\mathbf S}_0(n,r)$ is self-injective by \cite[Proposition 4.4]{DY12}, \cite[Section 1.6, Exercises 2]{91Benson} implies that the socle of the projective indecomposable module $P_\lambda$ is simple.
Therefore, $\mathrm{soc}(P_\lambda) \cong S_{\lambda^{\mathrm{r}}}$, as desired.
\end{proof}

\subsection{Proof of \cref{lem: phi}}\label{sec: proof of psi_nr}

We adopt the notation of \cite{DY12}, summarized as follows.

\begin{itemize}
  \item For $\lambda = (\lambda_1,\ldots,\lambda_n) \in \Lambda(n,r)$ and $1 \leq i \leq n$, set
  \[
  R_i^\lambda = \{\, x \mid \lambda_1 + \cdots + \lambda_{i-1} + 1 \leq x \leq \lambda_1 + \cdots + \lambda_i \,\},
  \]
  with the convention that $R_i^\lambda = \varnothing$ if $\lambda_i = 0$.  
  This yields a disjoint union decomposition
  \[
  \{1,2,\ldots,r\} = R_1^\lambda \cup R_2^\lambda \cup \cdots \cup R_n^\lambda.
  \]

  \item For each $A \in M_n(r)$, let $\lambda = \mathrm{row}(A)$ and $\mu = \mathrm{col}(A)$.  
  Choose $w_A \in \mathfrak{S}_r$ such that for $1 \leq i,j \leq n$,
  \[
  a_{i,j} = \lvert R_i^\lambda \cap (w_A R_j^\mu)\rvert.
  \]

  \item Let $d_A$ denote the shortest element in the double coset $\mathfrak{S}_\lambda w_A \mathfrak{S}_\mu$.

  \item Define
  \[
  \bar x_{\mathfrak{S}_\lambda w \mathfrak{S}_\mu} = \sum_{x \in \mathfrak{S}_\lambda w \mathfrak{S}_\mu} \bar\pi_x.
  \]

  \item For $\lambda,\mu \in \Lambda(n,r)$ and $w \in \mathfrak{S}_r$, define
  \[
  \phi_{\lambda,\mu}^w : \bigoplus_{\nu \in \Lambda(n,r)} \bar x_\nu H_0(r) \;\longrightarrow\;
  \bigoplus_{\nu \in \Lambda(n,r)} \bar x_\nu H_0(r), \qquad 
  \bar x_\nu h \longmapsto \delta_{\mu,\nu}\,\bar x_{\mathfrak{S}_\lambda w \mathfrak{S}_\mu}\,h.
  \]
\end{itemize}

By \cite[(2.1.2)]{DY12}, for each $A \in M_n(r)$ we have
\[
e_A = \phi^{d_A}_{\mathrm{row}(A),\,\mathrm{col}(A)}.
\]
For a word ${\bf i} = [i_1, \ldots, i_r]$ with $i_j \in \{1,2,\ldots,n\}$, set 
\[
{\bar \xi}_{\bf i} := {\bar \xi}_{i_1} \otimes \cdots \otimes {\bar \xi}_{i_r},
\]
and for $\lambda \in \Lambda(n,r)$, define
\[
{\bf i}_{\lambda} := [\underbrace{1,\ldots,1}_{\lambda_1}, \ldots, \underbrace{n,\ldots,n}_{\lambda_n}].
\]
For example, ${\bar \xi}_{{\bf i}_{(2,0,1)}} = {\bar \xi}_1 \otimes {\bar\xi_1} \otimes {\bar \xi}_3$.
The isomorphism
\[
\phi_0: \bigoplus_{\lambda \in \Lambda(n, r)} {\bar x}_\lambda {\mathbf H}_r(0) \stackrel{\cong}{\longrightarrow} {V_0}^{\otimes r}
\]
from \cref{isomorphism phi 0} is given by
\[
{\bar x}_\lambda \, h \mapsto {\bar \xi}_{{\bf i}_\lambda} \, h, \qquad h \in {\mathbf H}_r(0)
\]
(see \cite[Remark 2.5]{DY12}).

For $\lambda \in \Lambda(n,r)$, let $(V_0^{\otimes r})_{\lambda} \subseteq V_0^{\otimes r}$ be the $\mathbb{C}$-span of $\{{\bar \xi}_{\bf i} \mid \mathrm{cont}({\bf i}) = \lambda \}$, where 
\[
\mathrm{cont}({\bf i}) = (\#1 \text{ in } {\bf i}, \ldots, \#n \text{ in } {\bf i}).
\]
Then
\[
V_0^{\otimes r} = \bigoplus_{\lambda \in \Lambda(n,r)} (V_0^{\otimes r})_{\lambda},
\]
and $\phi_0$ restricts to an isomorphism
\[
\phi_0\mid_{{\bar x}_\lambda {\mathbf H}_r(0)}: {\bar x}_\lambda {\mathbf H}_r(0) \stackrel{\cong}{\longrightarrow} (V_0^{\otimes r})_{\lambda}.
\]

By conjugating $\phi_0$, we may regard each $e_A \in \mathbf{S}_0(n,r)$ as an element of $\operatorname{End}_{{\mathbf H}_r(0)}(V_0^{\otimes r})$.

\begin{lemma}\label{lem: k_lam act on V^r}
For $\lambda \in \Lambda(n,r)$ and ${\bf i} = [i_1,\ldots,i_r]$ a word with $i_j \in \{1,2,\ldots,n\}$, we have 
\[
k_{\lambda}({\bar \xi}_{\bf i}) = 
\begin{cases}
{\bar \xi}_{\bf i}, & \text{if $\mathrm{cont}({\bf i}) = \lambda$}, \\ 
0, & \text{otherwise}.
\end{cases}
\]
\end{lemma}

\begin{proof}
Let $A = D_{\lambda}$. Then $\row(A) = \col(A) = \lambda$, $w_A = \mathrm{id}$ and $d_A = \mathrm{id}$. Therefore 
\[
\bar x_{\mathfrak{S}_{\row(A)} w_A \mathfrak{S}_{\col(A)}} = {\bar x}_{\lambda},
\]
and thus 
\[
\phi^{d_A}_{\mathrm{row}(A),\,\mathrm{col}(A)} ({\bar x}_{\nu} \, h) = \delta_{\lambda, \nu} \, {\bar x}_{\lambda} \, h
\]
for any $\nu \in \Lambda(n,r)$ and $h \in {\mathbf H}_r(0)$. 
This proves the assertion.
\end{proof}

Let ${\bf i} = [i_1,\ldots,i_r]$ be a word with $i_j \in \{1,2,\ldots,n\}$.  
To study the elements $e_{i,\lambda}, f_{i,\lambda} \in S_0(n,r)$, we introduce the notions of $e_i$- and $f_i$-descendants associated with ${\bf i}$.

\begin{definition}
\leavevmode
\begin{itemize}
  \item A word ${\bf j} = [j_1,\ldots,j_r]$ with $j_k \in \{1,2,\ldots,n\}$ is called \emph{$e_i$-descendant of ${\bf i}$} the following conditions hold:
  \begin{itemize}
    \item There exists exactly one position $k$ such that $i_k = i+1$ and $j_k = i$.
    \item For all $\ell > k$, we have $i_\ell \neq i$.
    \item For all other positions $m \neq k$, we have $j_m = i_m$.
  \end{itemize}
  We denote by $E_i({\bf i})$ the set of all $e_i$-descendants of ${\bf i}$.

  \item A word ${\bf j} = [j_1,\ldots,j_r]$ with $j_k \in \{1,2,\ldots,n\}$ is called \emph{$f_i$-descendant of ${\bf i}$} the following conditions hold:
  \begin{itemize}
    \item There exists exactly one position $k$ such that $i_k = i$ and $j_k = i+1$.
    \item For all $\ell < k$, we have $i_\ell \neq i+1$.
    \item For all other positions $m \neq k$, we have $j_m = i_m$.
  \end{itemize}
  We denote by $F_i({\bf i})$ the set of all $f_i$-descendants of ${\bf i}$.
\end{itemize}
\end{definition}

For instance,
\begin{align*}
E_3([1,3,4,3,1,4,2,4]) &= 
\{\, [1,3,4,3,1,3,2,4],\ [1,3,4,3,1,4,2,3]\,\},\\
F_3([1,3,4,3,1,4,2,4]) &= 
\{\, [1,4,4,3,1,4,2,4]\,\}.    
\end{align*}

\begin{lemma}\label{lem: ei lam fi lam act on V^r}
Let $\lambda \in \Lambda(n,r)$, $1 \le i \le n-1$, and let ${\bf i} = [i_1,\ldots,i_r]$ be a word with $i_j \in \{1,2,\ldots,n\}$. Then
\[
e_{i,\lambda}({\bar \xi}_{\bf i}) = 
\begin{cases}
\sum_{{\bf j} \in E_i({\bf i})}{\bar \xi}_{\bf j}, & \text{if $\mathrm{cont}({\bf i}) = \lambda$}, \\ 
0, & \text{otherwise},
\end{cases}
\]
and
\[
f_{i,\lambda}({\bar \xi}_{\bf i}) = 
\begin{cases}
\sum_{{\bf j} \in F_i({\bf i})}{\bar \xi}_{\bf j}, & \text{if $\mathrm{cont}({\bf i}) = \lambda$}, \\ 
0, & \text{otherwise}.
\end{cases}
\]
\end{lemma}

\begin{proof}
Let $A = D_{\lambda} - E_{i+1,i+1} + E_{i,i+1}$. 
Set $\lambda^{i-} := (\lambda_1,\ldots,\lambda_i+1,\lambda_{i+1}-1,\ldots,\lambda_n)$. 
Then $\row(A) = \lambda^{i-}$, $\col(A) = \lambda$, and $w_A = \mathrm{id}$. Thus $d_A = \mathrm{id}$. 
Therefore
\[
\bar x_{\mathfrak{S}_{\row(A)} w_A \mathfrak{S}_{\col(A)}} 
= \bar x_{\mathfrak{S}_{\lambda^{i-}} \,\mathrm{id} \, \mathfrak{S}_{\lambda}} 
=  \sum_{w \in \mathfrak{S}_{\lambda^{i-}} \,\mathrm{id} \, \mathfrak{S}_{\lambda}} {\bar \pi}_w,
\]
which factors as
\[
\left( \sum_{w \in \mathfrak{S}_{\lambda^{i-}}} {\bar \pi}_w \right) 
\left( \sum_{w \in \prescript{\lambda^{i-}}{}{\mathfrak{S}}_{\lambda}} {\bar \pi}_w \right) 
= {\bar x}_{\lambda^{i-}} \left( \sum_{w \in \prescript{\lambda^{i-}}{}{\mathfrak{S}}_{\lambda}} {\bar \pi}_w \right),
\]
where
\[
\prescript{\mu}{}{\mathfrak{S}}_{\lambda} 
:= \{\, w \in \mathfrak{S}_{\lambda} \mid \ell(sw) > \ell(w)\ \text{for all simple reflections $s$ in $\mathfrak{S}_\mu$}\,\}.
\]

One verifies that
\[
\prescript{\lambda^{i-}}{}{\mathfrak{S}}_{\lambda} 
= \{\, x_0, x_1, \ldots, x_{\lambda_{i+1}-1} \,\},
\]
where $x_0 = \mathrm{id}$ and
\[
x_j = s_{b_i+1} s_{b_i+2} \cdots s_{b_i+j}
\qquad (1 \le j \le \lambda_{i+1}-1),
\]
with $b_i = \lambda_1 + \cdots + \lambda_i$.
Hence $e_A$ acts by 
\[
  {\bar \xi}_{{\bf i}_\nu} h \longmapsto \delta_{\lambda,\nu}\,{\bar \xi}_{{\bf i}_\lambda^{i-}} \left( \sum_{0 \le j \le \lambda_{i+1}-1} {\bar \pi}_{x_j} \right)\,h.
\]

For a word ${\bf i}$ with $\mathrm{cont}({\bf i}) = \lambda$ and $1 \le k \le \lambda_{i+1}$, define
\[
{\bf i}^{(i,k)} := \text{the word obtained from ${\bf i}$ by replacing the $k$-th occurrence of $i+1$ with $i$}.
\]
Then 
\[
{\bar \xi}_{{\bf i}_\lambda^{\,i-}} {\bar \pi}_{x_j}
= {\bar \xi}_{{\bf i}_{\lambda}^{(i,j+1)}},
\]
so that
\[
{\bar \xi}_{{\bf i}_\lambda^{\,i-}}
\left( \sum_{0 \le j \le \lambda_{i+1}-1} {\bar \pi}_{x_j} \right)
= \sum_{1 \le k \le \lambda_{i+1}} {\bar \xi}_{{\bf i}_{\lambda}^{(i,k)}}.
\]

If a word ${\bf i}$ satisfies $\mathrm{cont}({\bf i})=\lambda$, then there exists a unique permutation 
$w \in \mathfrak{S}_r$ such that
\[
{\bar \xi}_{\bf i} = {\bar \xi}_{{\bf i}_\lambda}\,{\bar \pi}_w.
\]
Here $w$ is the unique minimal-length permutation sending the ordered word 
\[
{\bf i}_\lambda = [\underbrace{1,\ldots,1}_{\lambda_1},\underbrace{2,\ldots,2}_{\lambda_2},\ldots,\underbrace{n,\ldots,n}_{\lambda_n}]
\]
to ${\bf i}$, and it can be obtained by successively applying simple transpositions at positions 
where the entries are strictly increasing (cf.~\cref{eq: Hecke right action on V^r q=0}).

Let $p:=b_i+k$ be the position of the $k$-th occurrence of $i{+}1$ in ${\bf i}_\lambda$.
In any reduced expression $w=s_{j_1}\cdots s_{j_t}$ adapted to 
${\bf i}_\lambda\!\to{\bf i}$ (i.e. each $s_{j_\alpha}$ swaps an increasing adjacent pair),
the letter $i{+}1$ starting at position $p$ moves left across an $i$
exactly as many times as there are $i$’s to the right of the $k$-th $i{+}1$ in ${\bf i}$.
Consequently:
\begin{itemize}
\item If there is \emph{no} $i$ to the right of the $k$-th $i{+}1$ in ${\bf i}$, then this letter never
crosses an $i$ along $w$, and hence
\[
{\bar \xi}_{{\bf i}_{\lambda}^{(i,k)}}\,{\bar \pi}_w \;=\; {\bar \xi}_{{\bf i}^{(i,k)}}.
\]
\item If there \emph{is} at least one $i$ to its right, then along $w$ there is a first swap that moves
this $i{+}1$ past an $i$. Starting instead from ${\bar \xi}_{{\bf i}_{\lambda}^{(i,k)}}$, the two
entries at that position are $(i,i)$, so that factor of ${\bar \pi}$ kills the vector by the
equal–entries branch of \cref{eq: Hecke right action on V^r q=0}; hence
\[
{\bar \xi}_{{\bf i}_{\lambda}^{(i,k)}}\,{\bar \pi}_w \;=\; 0.
\]
\end{itemize}
This proves the claim for $e_{i,\lambda}$, and the statement for $f_{i,\lambda}$ follows analogously.
\end{proof}

\begin{proof}[Proof of \cref{lem: phi}]
From the coproduct formulas of $U_0(\mathfrak{gl}_n)$ one obtains, for $r \geq 1$ and $1 \leq i \leq n-1$,
\begin{align*}
\Delta^{(r)}(\mathbf{k}_i)
&= \mathbf{k}_i^{\otimes r},\\[4pt]
\Delta^{(r)}(\mathbf{e}_i)
&= \sum_{t=1}^{r}\;
\underbrace{1\otimes\cdots\otimes 1}_{t-1}\ \otimes\ \mathbf{e}_i\ \otimes\ 
\underbrace{\mathbf{k}_i\otimes\cdots\otimes \mathbf{k}_i}_{\,r-t\,},\\[4pt]
\Delta^{(r)}(\mathbf{f}_i)
&= \sum_{t=1}^{r}\;
\underbrace{\mathbf{k}_{i+1}\otimes\cdots\otimes \mathbf{k}_{i+1}}_{t-1}\ \otimes\ \mathbf{f}_i\ \otimes\ 
\underbrace{1\otimes\cdots\otimes 1}_{\,r-t\,}.
\end{align*}

It follows that the action on a simple tensor ${\bar \xi}_{\bf i}$ is given by 
\begin{align*}
    \mathbf{k}_i \cdot {\bar \xi}_{\bf i} &= 
    \begin{cases}
        {\bar \xi}_{\bf i}, & \text{if ${\bf i}$ contains no entry equal to $i$}, \\ 
        0, & \text{otherwise},
    \end{cases}\\[4pt]
    \mathbf{e}_i \cdot {\bar \xi}_{\bf i} & = \sum_{{\bf j} \in E_i({\bf i})} {\bar \xi}_{\bf j},\\[4pt]
    \mathbf{f}_i \cdot {\bar \xi}_{\bf i} & = \sum_{{\bf j} \in F_i({\bf i})} {\bar \xi}_{\bf j}.
\end{align*}

Comparing these identities with  
\cref{lem: k_lam act on V^r} and \cref{lem: ei lam fi lam act on V^r}, the desired identities follow. 
\end{proof}

We end this subsection with an easy but important corollary of \cref{lem: k_lam act on V^r}. 

\begin{corollary} \label{polynomial modules are Schur modules} The following hold:

\begin{enumerate}[label = {\rm (\arabic*)}]
   \item 
Let $\mathcal C$ be the full subcategory of $U_0(\mathfrak{gl}_n)\text{-\textsf{mod}}$  consisting of polynomial $U_0(\mathfrak{gl}_n)$-modules of degree $r$.
Then there exists a functor 
$$\xi_{n,r}: \mathcal C \to \mathbf{S}_0(n,r)\text{-\textsf{mod}}$$
such that $\Psi_{n,r}\circ \xi_{n,r}= {\rm id}$.

\item 
The functor $\xi_{n,r}$ preserves the socle; more precisely, for every polynomial $U_0(\mathfrak{gl}_n)$-module $M$ of degree $r$, 
\[\xi_{n,r}(\operatorname{soc}(M))
= \operatorname{soc}(\xi_{n,r}(M)).
\]

\item 
On the image of $\xi_{n,r}$, the functor $\Psi_{n,r}$ preserves the socle;; more precisely, for every polynomial $U_0(\mathfrak{gl}_n)$-module $M$ of degree $r$, 
\[
 \Psi_{n,r}(\operatorname{soc}(\xi_{n,r}(M)))= \operatorname{soc}(M).
\]
In particular, for each $\alpha \in \Lambda^+(r)$,
we have 
\[\operatorname{soc}(N_{\alpha})
=\Psi_{n,r}(\operatorname{soc}(G_{n,r}(R_\alpha))).
\]
If $\ell(\alpha)\le n$, then 
$\operatorname{soc}(N_{\alpha})
= D_{\alpha^{\rmr}}$.

\end{enumerate}
\end{corollary}

\begin{proof}
(1) 
For each $\lambda \in \Lambda(n,r)$, the idempotent $k_\lambda \in \mathbf{S}_0(n,r)$ acts on $V_0^{\otimes r}$ as in \cref{lem: k_lam act on V^r}. 
Since $\{k_\lambda \mid \lambda \in \Lambda(n,r)\}$ is a set of orthogonal idempotents of $\mathbf{S}_0(n,r)$ with $\sum_{\lambda \in \Lambda(n,r)}k_\lambda=1$,
we obtain
\[
M = \bigoplus_{\lambda \in \Lambda(n,r)} M_\lambda, \qquad M_\lambda := k_\lambda(M),
\]
so each $k_\lambda$ also acts on $M$.
By \cref{lem: phi},
\[
\psi_{n,r}(\mathbf{e}_i) = e_i,\qquad 
\psi_{n,r}(\mathbf{f}_i) = f_i,\qquad 
\psi_{n,r}(\mathbf{k}_i)=\sum_{\mu_i=0}k_\mu.
\]
Hence, on the same underlying vector space $M$, we define the actions of $e_i,f_i$ in the same way as $\mathbf{e}_i,\mathbf{f}_i$. Together with the $k_\lambda$-action described above, these operators endow $M$ with a natural $\mathbf{S}_0(n,r)$-module structure. Define $\xi_{n,r}(M)$ to be the resulting $\mathbf{S}_0(n,r)$-module. 
By construction,
we have $\Psi_{n,r}(\xi_{n,r}(M)) = M$.

(2) 
In general, $\operatorname{soc}(M)$ is contained in $\Psi_{n,r}(\operatorname{soc}(\xi_{n,r}(M)))$. If the inclusion is strict, then by (1) the image $\xi_{n,r}(\operatorname{soc}(M))$ contains $\operatorname{soc}(\xi_{n,r}(M))$, which yields a contradiction.

(3) The first assertion follows from (2).
The second assertion can be obtained by combining this result with \cref{socle of Pim of Schur}.  
\end{proof}

\section{Symmetries of the Cartan matrix of ${\mathbf S}_0(n,r)$}
\label{main result: 3}
The Cartan matrix of ${\mathbf S}_0(n,r)$ is defined by the matrix 
$$
C = (c_{\lambda, \mu})_{\lambda, \mu \in \Lambda^{\bullet}(n,r)},
$$
where $c_{\lambda, \mu} = \dim_\mathbb{C} \Hom_{{S_0(n,r)}_\mathbb{C}}(P_{\lambda}, P_{\mu})$.

\begin{lemma} {\rm \cite[1.7.6 and 1.7.7]{91Benson}}
\label{multiplicity of composition factors}
Let $A$ be a finite dimensional algebra over $\mathbb C$ and $M$ be a finite dimensional $A$-module.
\begin{enumerate}[label = {\rm (\arabic*)}]
\item 
Let $P$ be a projective indecomposable $A$-module. Then
$
\dim_{\Delta} \operatorname{Hom}_{A}(P, M)
$
equals the multiplicity $[M : \operatorname{top}(P)]$ of $\operatorname{top}(P)$ as a composition factor of $M$, where 
$\Delta=\operatorname{End}_A(\operatorname{top}(P))$.

\item 
Let $I_S$ denote the injective hull of a simple $A$-module $S$ and  
$\Delta=\operatorname{End}_A(S)$. Then
$
\dim_{\Delta} \operatorname{Hom}_{A}(M, I_S)
$
equals the multiplicity $[M : S]$ of $S$ as a composition factor of $M$.
\end{enumerate}
\end{lemma}

Since $\mathbb C$ is algebraically closed, 
$\End_{{S_0(n,r)}_\mathbb C}(S_{\lambda}) \cong \mathbb C$ as $\mathbb C$-spaces for all $\lambda \in \Lambda^{\bullet}(n,r)$. Thus, by \cref{multiplicity of composition factors} (1), $c_{\lambda, \mu}$ equals $[P_{\mu}: S_{\lambda}]$.
For each $\mu \in \Lambda(n,r)$, define
\begin{equation*}
B_{\mu} := \left\{ e_A \mid \ro(A) = \mu \right\} \setminus \bigcup_{\alpha \in \mathrm{max}(\mu) \setminus \{(1,1,\ldots,1)\}} B_{\mu, \alpha},
\end{equation*}
where
\[
B_{\mu, \alpha} := \left\{ e_A \mid \text{$\ro(A) = \mu$ and $A$ is open on rows with respect to $\alpha$} \right\}
\]
(see \cite[Section 5.2]{JSY16} for the precise definition).
With this notation, it was shown in \cite[Proposition 5.3]{JSY16} that 
    \[
    c_{\lambda, \mu}= |B^{\lambda} \cap B_{\mu}|.
    \]

The purpose of this section is to introduce certain symmetries of the Cartan matrix $C$.
To this end, it is necessary to recall the symmetries of the Cartan matrix of the $0$-Hecke algebra $\mathbf {H}_0(r)$.

Let 
$$D=(d_{\alpha, \beta})_{\alpha, \beta \in \Lambda^+(r)}$$
denote the Cartan matrix of $\mathbf {H}_0(r)$, that is,  
$d_{\alpha, \beta}=\dim_\mathbb C {\rm Hom}_{\mathbf {H}_0(r)}(R_\alpha, R_\beta)$.

Let $f: \mathbf {H}_0(r) \to \mathbf {H}_0(r)$ be an automorphism. For any $\mathbf {H}_0(r)$-module $M$, let $f[M]$ be the $\mathbf {H}_0(r)$-module with the same underlying vector space as $M$, where the action is twisted by $f$ via
$$
b \cdot_f v := f(b) \cdot v \quad \text{for all } b \in \mathbf {H}_0(r),\ v \in M.
$$
Similarly, given an anti-isomorphism $g: \mathbf {H}_0(r) \to \mathbf {H}_0(r)$, let $g[M]$ be the $\mathbf {H}_0(r)$-module with underlying space $M^*$, the dual of $M$, and action $\cdot^g$ defined by
\begin{align}\label{eq: anti-automorphism twist}
(b \cdot^g \delta)(v) := \delta(g(b) \cdot v) \quad \text{for all } b \in \mathbf {H}_0(r), \delta \in M^*, v \in M.
\end{align}
Let
\begin{align}\label{eq: isomorphism twist}
\mathbf{T}^+_f \colon \mathbf {H}_0(r)\text{-mod} \to \mathbf {H}_0(r)\text{-mod}, \quad M \mapsto f[M]
\end{align}
be the covariant functor induced by the automorphism $f$, and
\begin{align}\label{eq: anti-isomorphism twist}
\mathbf{T}^-_g \colon \mathbf {H}_0(r)\text{-mod} \to \mathbf {H}_0(r)\text{-mod}, \quad M \mapsto g[M]
\end{align}
be the contravariant functor induced by the anti-isomorphism $g$.
The functors $\mathbf{T}^+_f$ and $\mathbf{T}^-_g$ are called the \emph{$f$-twist} and \emph{$g$-twist}, respectively. 

For the (anti-)involutions $\autophi, \autotheta$, and $\autochi$ in \cref{automorphism of Fayers}, the corresponding twists on simple modules and projective indecomposable modules are listed in \cref{tab: auto twist of F and P}.\footnote{For additional (anti-)involutions, see \cite[Table 2, Section 3.4]{22JKLO}.} 
These (anti-)involution twists yield the following symmetries of the Cartan matrix of $\mathbf{H}_0(r)$.

\begin{table}[t]
\fontsize{9}{9}
\renewcommand*\arraystretch{1.2}
\begin{tabular}{ c || c | c | c } 
& $\autophi$-twist
& $\autotheta$-twist
& $\autochi$-twist
\\ \hline \hline
$F_\alpha$
& $F_{\alpha^\rmr}$
& $F_{\alpha^\rmc}$
& $F_{\alpha}     $
\\ \hline
$R_\alpha$
& $R_{\alpha^\rmr}$
& $R_{\alpha^\rmc}$
& $R_{\alpha^\rmr}$
\end{tabular}
\caption{(Anti-)involution twists of $F_\alpha$ and $R_\alpha$}
\label{tab: auto twist of F and P}
\end{table}

\begin{lemma} \label{symmetries of H cartan}
For $\alpha, \beta \in \Lambda^+(r)$, the following hold.
\begin{enumerate}[label = {\rm (\arabic*)}]
\item 
$d_{\alpha,\beta}=d_{\alpha^{\rmr}, \beta^{\rmr}}$,
$d_{\alpha,\beta}=d_{\beta^{\rmr}, \alpha^{\rmr}}$,
and  
$d_{\alpha,\beta}=d_{\beta, \alpha}$
\item 
$d_{\alpha,\beta}=d_{\alpha^{\rmc}, \beta^{\rmc}}$
\item 
$d_{\alpha,\beta}=d_{\beta^{\rmr}, \alpha}$
\end{enumerate}
\end{lemma}

\begin{proof}
(1) Combining \cref{tab: auto twist of F and P} with the fact that $\mathbf{T}^+_\autophi$ is an equivalence of categories, we obtain $\mathbb C$-vector space isomorphisms   
$${\rm Hom}_{H_r(0)}(R_\alpha, R_\beta) \cong 
{\rm Hom}_{H_r(0)}(\mathbf{T}^+_\autophi(R_\alpha), \mathbf{T}^+_\autophi(R_\beta))\cong {\rm Hom}_{H_r(0)}(R_{\alpha^{\rmr}}, R_{\beta^{\rmr}}).$$
Combining \cref{tab: auto twist of F and P} with the fact that $\mathbf{T}^-_\chi$ is a duality functor, we obtain $\mathbb C$-vector space isomorphisms   
$$ {\rm Hom}_{H_r(0)}(R_\alpha, R_\beta)\cong 
 {\rm Hom}_{H_r(0)}(\mathbf{T}^-_\chi(R_\beta), \mathbf{T}^-_\chi(R_\alpha)) \cong  {\rm Hom}_{H_r(0)}(R_{\beta^{\rmr}}, R_{\alpha^{\rmr}}).$$
Finally, by combining (1) and (2), we see that $d_{\alpha, \beta}=d_{\beta,\alpha}$.
From the representation-theoretical viewpoint,it can be derived from $\mathbb C$-vector space isomorphisms 
$${\rm Hom}_{H_r(0)}(R_\alpha, R_\beta) \cong 
{\rm Hom}_{H_r(0)}(\mathbf{T}^-_{\autochi \circ \autophi}(R_\alpha), \mathbf{T}^-_{\autochi \circ \autophi}(R_\beta))\cong {\rm Hom}_{H_r(0)}(R_{\beta}, R_{\alpha}).$$

(2) Combining \cref{tab: auto twist of F and P} with the fact that $\mathbf{T}^+_\autotheta$ is an equivalence, we obtain $\mathbb C$-vector space isomorphisms   
$$ {\rm Hom}_{H_r(0)}(R_\alpha, R_\beta)\cong 
 {\rm Hom}_{H_r(0)}(\mathbf{T}^+_\autotheta(R_\alpha), \mathbf{T}^+_\autotheta(R_\beta)) \cong  {\rm Hom}_{H_r(0)}(R_{\alpha^{\rmc}}, R_{\beta^{\rmc}}).$$
 
(3) Since $R_{\beta}$ is the injective hull of $F_{\beta^{\rmr}}$, it follows from \cref{multiplicity of composition factors} (2) that  
$$d_{\alpha, \beta}=\dim_\mathbb C {\rm Hom}_{H_r(0)}(R_\alpha, R_\beta)=[R_\alpha:F_{\beta^{\rmr}}]$$
On the other hand, \cref{multiplicity of composition factors} (1) says that 
$$[R_\alpha:F_{\beta^{\rmr}}]=\dim_\mathbb C {\rm Hom}_{H_r(0)}(R_{\beta^{\rmr}}, R_\alpha)=d_{\beta^{\rmr}, \alpha}.$$ 
\end{proof}

Define bijective maps
\begin{align*}
{\mathsf {ex}} &: \Lambda^+(r)^2 \to  \Lambda^+(r)^2, \quad (\alpha, \beta) \mapsto (\beta, \alpha)\\
{\mathsf {c}}&: \Lambda^+(r)^2 \to  \Lambda^+(r)^2, \quad (\alpha, \beta) \mapsto (\alpha^\rmc, \beta^\rmc )\\
\eta &: \Lambda^+(r)^2 \to  \Lambda^+(r)^2, \quad (\alpha, \beta) \mapsto (\beta^\rmr, \alpha).
\end{align*}
Then $\eta^2((\alpha, \beta))=((\alpha^{\rmr}, \beta^{\rmr}))$.
The orders of ${\mathsf {ex}}$ and ${\mathsf {c}}$ divide $2$ and that of $\eta$ divides $4$.
Furthermore, they satisfy the relations
$$\eta \circ {\mathsf {ex}}={\mathsf {ex}} \circ \eta^3, \quad {\mathsf {ex}}\circ {\mathsf {c}}={\mathsf {c}} \circ {\mathsf {ex}}, \quad \eta \circ {\mathsf {c}}={\mathsf {c}} \circ \eta. $$
This implies that there is a surjection from $D_4 \times \mathbb Z_2$ to $<{\mathsf {ex}},{\mathsf {c}},\eta>$, where $D_4$ is the dihedral group of order $8$.
Thus $D_4 \times \mathbb Z_2$ acts on  $\Lambda^+(r)^2$.

\begin{example}
(1) Let $r=4$. 
In this case, $<{\mathsf {ex}},{\mathsf {c}},\eta>\cong D_4 \times \mathbb Z_2$, and 
$\Lambda^+(4)^2$ has exactly 64 elements grouped into 15 orbits (see \cref{orbit dist}).

{\small
\begin{table}[h]
\centering
\begin{tabular}{|c|c|c|}
\hline
\text{ Orbit size} & \text{Number of orbits} & \text{Total elements} \\
\hline
1  & 1 & 1 \\
2  & 3 & 6 \\
4  & 7 & 28 \\
8  & 4 & 32 \\
\hline
\text{Total} & \text{12 orbits} & \text{64 elements} \\
\hline
\end{tabular}
\caption{Orbit decomposition of $\Lambda^+(4)^2$ under the action of $\langle \mathsf{ex}, \mathsf{c}, \eta \rangle$}
\label{orbit dist}
\end{table}
}

\noindent
For instance, the unique orbit of size 1 is $\{((2,2), (2,2))\}$, and  
the orbit of 
$((3,1), (1,1,2))$ is  
\begin{align*}
&((3,1), (1,1,2)) \stackrel{\eta}{\to} ((2,1,1), (3,1))
\stackrel{\eta}{\to}((1,3), (2,1,1))\stackrel{\eta}{\to}((1,1,2), (1,3))\stackrel{{\mathsf {ex}}}{\to}\\ 
&
((1,3),(1,1,2))\stackrel{\eta}{\to}
((2,1,1),(1,3))\stackrel{\eta}{\to}
((3,1),(2,1,1))\stackrel{\eta}{\to}((1,1,2),(3,1)). 
\end{align*}

(2) When $r=6$, one can find an orbit consisting of 16 elements; for instance, the orbit of $((1,2,3), (1,3,2))$ has cardinality 16.  
\end{example}

With the above group action, we 
derive the following theorem
from \cref{symmetries of H cartan}.

\begin{proposition}\label{symmetry of Catran matrix of 0 Hecke algebra}
The Cartan matrix of ${\mathbf S}_0(n,r)$ is constant on each $D_4 \times \mathbb{Z}_2$-orbit. More precisely, for every pair $(\alpha, \beta)$ in the same $D_4 \times \mathbb{Z}_2$-orbit, the entry $d_{\alpha, \beta}$ takes the same value.
\end{proposition}

In \cref{relation between weak and strong compositions}, we introduced the bijections
\begin{align*} 
& \Lambda^{\bullet}(n,r) \xrightarrow{\ \sim\ } \bigsqcup_{0 \le \ell \le n} \Lambda^+(\ell, r), \quad \lambda \mapsto \lambda^+,\\
& \bigsqcup_{0 \le \ell \le n} \Lambda^+(\ell, r) \xrightarrow{\ \sim\ } \Lambda^{\bullet}(n,r), \quad \alpha \mapsto \alpha^{\bullet}.
\end{align*}
Via these bijections, we regard 
$\Lambda^{\bullet}(n,r)$ as a subset of $\Lambda^+(\ell, r)$.

\begin{theorem} \label{symmetries of 0 Schur cartan}
For $0\le n\le r$, the Cartan matrix 
$$C = (c_{\lambda, \mu})_{\lambda, \mu \in \Lambda^{\bullet}(n,r)}$$ 
of ${\mathbf S}_0(n,r)$
is obtained from  
the Cartan matrix 
$$D=(d_{\alpha, \beta})_{\alpha, \beta \in \Lambda^+(r)}$$ of ${\mathbf H}_r(0)$
by restricting the index set from $\Lambda^+(r)$
to $\Lambda^{\bullet}(n,r)$.
In particular, the two matrices coincide when $n = r$.

\end{theorem}

\begin{proof}
Let 
\[
0 = M_0  \subset M_1 \subset M_2 \subset \cdots  \subset M_s = R_{\mu^+}
\]
be a composition series of $R_{\mu^+}$.
Since the functor $G_{n,r}: {\mathbf H}_0(r)\text{-mod} \to {\mathbf S}_0(n,r)\text{-mod}$ is exact, it follows from \cref{cor: 0-Schur functor 0-Hecke} that 
\[
0 = G_{n,r}(M_0)  \subset G_{n,r}(M_1) \subset G_{n,r}(M_2) \subset \cdots  \subset G_{n,r}(M_s)= P_{\mu}
\]
is a composition series of $P_{\mu}$.
Furthermore,  
\begin{align*}
 G_{n,r}(M_i)/ G_{n,r}(M_{i-1}) \cong 
 \begin{cases}
S_{\tau^\bullet} & \text{ if } \ell(\tau)\le n\\ 
0 & \text{ if } \ell(\tau)> n.  
 \end{cases}
\end{align*}
This tells us that $c_{\lambda,\mu}=d_{\lambda^+,\mu^+}$
completing the proof.
\end{proof}

Note that $\Lambda^\bullet(n,r)^2$ is not closed under the map ${\mathsf {c}}$, in general. 
So we only consider the group $<{\mathsf {ex}},\eta>$, which is a quotient group of $D_4$. 
Thus $D_4$ acts on $\Lambda^\bullet(n,r)^2$.

\begin{example}
Consider the case where $r=3$. Then $<{\mathsf {ex}},{\mathsf {c}},\eta>$-orbit of $((21), 3)$ in $\Lambda^+(3)^2$ is  
\begin{align*}
&((21), 3) \stackrel{\eta}{\to} (3, (21))\stackrel{\eta}{\to} ((12), 3)\stackrel{\eta}{\to}(3, (12))\\
& \hskip 8mm \downarrow {\mathsf {c}}\\
&((12), 111) \stackrel{\eta}{\to} (111, (12))\stackrel{\eta}{\to} ((21), 111) \stackrel{\eta}{\to}(111, (21)).
\end{align*}
On the other hand, $<{\mathsf {ex}},\eta>$-orbit of $((21), 3)$ in $\Lambda^\bullet(2,3)^2$ is
\begin{align*}
((21), (3)) \stackrel{\eta}{\to} ((3), (21))\stackrel{\eta}{\to} ((12), (3))\stackrel{\eta}{\to}((3), (12)).
\end{align*}
\end{example}

\section{Further avenues}\label{sec: Further remark}

\begin{enumerate}[label = {\rm (\arabic*)}]
\item 
We have an $\mathbf{S}_0(n,r)$-module isomorphism  
\[
\mathbf{S}_0(n,r)\cong \bigoplus_{\lambda\in\Lambda^\bullet(n,r)} P_\lambda^{\dim S_\lambda},
\]
and hence
\[
\operatorname{rad}(\mathbf{S}_0(n,r))\cong 
\bigoplus_{\lambda\in\Lambda^\bullet(n,r)} N_\lambda^{\dim S_\lambda}
\]
(see \cref{Describing radical }). However, despite this structural decomposition, an explicit description of 
\(\operatorname{rad}(\mathbf{S}_0(n,r))\) as a subspace of 
\(\mathbf{S}_0(n,r)\) is not available. As in the case of $N_\lambda$, 
it would be nice to have a combinatorial basis for 
\(\operatorname{rad}(\mathbf{S}_0(n,r))\).

\item 
If $\ell(\alpha) > n$, 
then $G_{n,r}(R_\alpha)$ and $N_\alpha$
are no longer projective.  
It would be nice to describe their projective covers and injective hulls.
In particular, by \cref{polynomial modules are Schur modules}, the socle of $G_{n,r}(R_\alpha)$ is mapped to that of $N_\alpha$
under the functor $\Psi_{n,r}$.

\medskip
\item 
The symmetries of the Cartan matrix of $\mathbf{H}_0(r)$ in \cref{symmetries of H cartan}~(1) arise from applying (anti-)involution twists. 
Hence, it is natural to expect that an analogous phenomenon may occur for the Cartan matrix of $\mathbf{S}_0(n,r)$.
\end{enumerate}

\medskip

\noindent {\bf Funding statement.}
The first author was supported by Basic Science Research Program through the National Research Foundation of Korea(NRF) funded by the Ministry of Education(No. 2024-00451844).
The second author was supported by NRF grant funded by the Korea government(MSIT) (No. RS-2024-00342349).
\medskip

\bibliographystyle{abbrv}
\bibliography{references}

\end{document}